\documentclass[11pt]{amsart}
\usepackage{amsmath} 
\usepackage{amsthm} 
\usepackage{graphicx} 
\usepackage{multicol} 
 \usepackage{multirow}
\usepackage{diagbox}
\usepackage{amssymb}

\usepackage{hyperref}
\usepackage{pst-node}
\usepackage{tikz-cd}
\usepackage{tikz}
\usetikzlibrary{calc}
\usepackage[mathscr]{euscript}
\usepackage[numbers]{natbib}
\usepackage[toc,page]{appendix}
\usepackage{dsfont}

\newtheorem{prop}{Proposition}
\newtheorem{thm}{Theorem}[section]
\newtheorem{lem}[thm]{Lemma}

\newtheorem{cor}[thm]{Corollary}
\newtheorem{example}[thm]{Examples}

\newtheorem{cnj}[thm]{Conjecture}
\theoremstyle{definition}

\theoremstyle{remark}

\usepackage{url,times,wasysym,geometry,indentfirst,rotating,tikz}

\newcommand{\mcA}{\mathcal{A}}
\newcommand{\mcB}{\mathcal{B}}

\newcommand{\mcM}{\mathcal{M}}

\newcommand{\mcR}{\mathcal{R}}

\newcommand{\mcV}{\mathcal{V}}

\newcommand{\mbbC}{\mathbb{C}}

\newcommand{\mbbQ}{\mathbb{Q}}

\newcommand{\mbbU}{\mathbb{U}}

\newcommand{\mbbZ}{\mathbb{Z}}

\title[Abelian Anyon Models]{In and around abelian Anyon Models}
\date{}
\author{Liang Wang}
\email{lwangmath@pku.edu.cn}
\address{Department of Mathematics, Peking University, Beijing 100871, China}
\author{Zhenghan Wang}
\email{zhenghwa@microsoft.com;zhenghwa@math.ucsb.edu}
\address{Microsoft Station Q and Department of Mathematics,
    University of California,
    Santa Barbara, CA 93106,
    U.S.A.}
    
\begin{document}

\begin{abstract}

Anyon models are algebraic structures that model universal topological properties in topological phases of matter and can be regarded as mathematical characterization of topological order in two spacial dimensions.  It is conjectured that every anyon model, or mathematically unitary modular tensor category, can be realized as the representation category of some chiral conformal field theory, or mathematically vertex operator algebra/local conformal net.  This conjecture is known to be true for abelian anyon models providing support for the conjecture.
We reexamine abelian anyon models from several different angles.  First anyon models are algebraic data for both topological quantum field theories and chiral conformal field theories.  While it is known that each abelian anyon model can be realized by a quantum abelian Chern-Simons theory and chiral conformal field theory, the construction is not algorithmic.  Our goal is to provide such an explicit algorithm for a $K$-matrix in Chern-Simons theory and a positive definite even one for a lattice conformal field theory.  Secondly anyon models and chiral conformal field theories  underlie the bulk-edge correspondence for topological phases of matter.  But there are interesting subtleties in this correspondence when stability of the edge theory and topological symmetry are taken into consideration.  Therefore, our focus is on the algorithmic reconstruction of extremal chiral conformal field theories with small central charges.  Finally we conjecture that a much stronger reconstruction holds for abelian anyon models: every abelian anyon model can be realized as the representation category of some non-lattice extremal vertex operator algebra generalizing the moonshine realization of the trivial anyon model.

\end{abstract}
\thanks{Z.W. is partially supported by NSF grant
FRG-1664351 and ARO MURI contract W911NF-20-1-0082.  We thank M. Bischoff,  T. Gannon, C. Meng, and E. Rowell for helpful comments and discussions.}
\maketitle

\section{Introduction}

The existence of anyons has received the strongest experimental support recently for abelian anyons in fractional quantum Hall (FQH) liquids \cite{B+20, Nakamura20}.  But the experiments to probe anyons in the bulk are not direct, instead measurements are made for edge states, which are modeled by the chiral conformal field theories ($\chi$CFTs) via the bulk-edge correspondence.
One main physical motivation for this work is the physical question how to realize abelian anyon models by FQH states in the Chern-Simons formalism and the corresponding edge physics of $\chi$CFTs, famously articulated as edge physics is the window to look into the bulk \cite{Wen92}.  The mathematical question is that given an abelian anyon model, how to algorithmically find a $K$-matrix for a Chern-Simons theory realization and a corresponding positive definite one for a corresponding lattice $\chi$CFT realization (with a small central charge).  We answered both questions in this paper and further refine the discussion by going beyond lattice $\chi$CFT realizations.  Non-lattice $\chi$CFT realization of anyon models can be regarded as a generalization of the Monstrous Moonshine module for the trivial anyon model, which could be regarded as the case of integer quantum Hall states.  Moreover, the realization of abelian anyon models by chiral CFTs is another positive instance of the realization conjecture for modular tensor categories in general \cite{TW17}.

Besides FQH physics, the $K$-matrix and $\chi$CFT for abelian anyon models also play important roles in the study of topological insulators and quantum cellular automata via the Walker-Wang model and systematical construction of fracton models.  There is a well-known subtlety in using abelian anyon models for quantum Hall physics: quantum Hall liquids are electron systems, while anyon models are for bosonic/spin systems.  The situation is well-understood and we will regard anyon models as the statistics sectors of anyon physics and the constituent electron is modelled by a charge sector.  The gluing of the statistics sectors and the charge sectors can be non-trivial, which is important for the detection of anyon statistics via edge current measurement \cite{CFW96}.  We will not further touch upon this subtlety in this paper.

The physics of the FQH effect continues to drive the development of topological physics.  It is not too much an exaggeration that topological physics is the reproduction of FQH effect in other physics systems.  A salient feature of FQH effect is the prediction of anyons as strongly supported by steady progress of experiments including the most recent ones \cite{B+20, Nakamura20}. The experimental probe of the bulk anyons is through the edge physics.  This remarkable bulk-edge correspondence between bulk anyon model and edge chiral conformal field theory deserves and inspires continued theoretical exploration.

The most robust and experimentally studied quantum Hall samples are those for filling fraction $\nu=1/3$.  The topological model of the bulk is via the anyon model $\mathbb{Z}_3\times sVec$, where we will refer to $\mathbb{Z}_3$ as the statistics sector and $sVec=\{1,f\}$ models the physical electron $f$, which was referred to as the charge sector.  It is well-established that the statistics sector of FQH is modeled by an anyon model/CS theory while the electron wave functions are conformal blocks of the edge conformal field theory as the famous Moore-Read state exemplified \cite{Wen92, MR}.  Further study of abelian anyon models will deepen our understanding on the one hand the physical study of edge physics of FQH states and on the other hand the general physics of bulk-edge correspondence.

Anyon models are algebraic structures that model universal topological properties in topological phases of matter and can be regarded as mathematical characterization of topological order in two spacial dimensions (e.g. see \cite{RW18}).  It is conjectured that every anyon model, or mathematically unitary modular tensor category, can be realized as the representation category of some chiral conformal field theory ($\chi$CFT), or mathematically vertex operator algebra (VOA)/local conformal net (LCN) \cite{TW17}. This conjecture is known to be true for abelian anyon models providing support for the conjecture, where one proof can be found in \cite{Nik1} and a more recent one in \cite{EG20}.  We reexamine abelian anyon models from several different angles.  First anyon models are algebraic data for both topological quantum field theories (TQFTs) and $\chi$CFTs.  While it is known that each abelian anyon model can be realized by a quantum abelian Chern-Simons (CS) TQFT and $\chi$CFT, the construction is not algorithmic (see e.g. \cite{Nik1, EG20, Cano14} and the references therein).  Our goal is to provide such an explicit algorithm for a $K$-matrix in CS theory and a positive definite one for a lattice $\chi$CFT.  Secondly anyon models and $\chi$CFTs underlie the bulk-edge correspondence for topological phases of matter.  But there are interesting subtleties in this correspondence when stability of the edge theory and topological symmetry are taken into consideration.  Therefore, our focus is on the algorithmic reconstruction of extremal $\chi$CFTs with small central charges.  Finally we conjecture that a much stronger reconstruction holds for abelian anyon models: every abelian anyon model can be realized as the representation category of some non-lattice extremal VOA generalizing the moonshine realization of the trivial anyon model.

Abelian anyon models are classified by metric groups---pairs $(A,\theta)$, where $A$ is a finite abelian group and $\theta(x)=e^{2 \pi i q(x)}, x\in A$ a non-degenerate quadratic form on $A$ \cite{JS93}.  Such a classification follows from two fundamental results: one is the classification of braided fusion categories by pre-metric groups, and the second is the classification of quadratic forms on finite abelian groups \cite{Wall63, Durfee}.  The full data of an abelian anyon model can be organized using the theory of abelian cohomology (see e.g. \cite{Galindo16}). 

The construction of an abelian CS theory for each abelian anyon model follows from the existence of a $K$-matrix, i.e., an even non-degenerate integral symmetric matrix.  The construction of a VOA for an abelian anyon model is reduced to the construction of a positive-definite ($+$-definite in the sequel) $K$-matrix regarded as a lattice.  The construction of a lattice VOA from such a lattice is part of the original work on VOAs realizing the moonshine module \cite{Borcherds86}.

Every anyon model corresponds to a unitary $(2+1)$-TQFT.  For abelian anyon models, the corresponding TQFTs are physically quantum abelian CS theories indexed by $K$-matrices.  We provide an algorithmic construction of such a $K$-matrix for each abelian anyon model.  The existence of such a $K$-matrix is known \cite{Nik1} and our refinement includes first an algorithmic constructin, and secondly the chiral central charge of the resulting CS theory is very small, which is important for physical applications.  We provide both a general method and concrete solutions for each case except the $B_{p^r}$ family when $p=1$ mod $8$\footnote{An explicit solution to this case depends on writing down a general quadratic non-residue for all $p$.}.  Our result relies on the classification of abelian anyon models and an adaption of the algorithm for constructing bilinear forms on finite abelian groups to $K$ matrices in \cite{Wall63}.

Gannon and the second author independently conjectured that every unitary modular tensor category or anyon model $\mathcal{B}$ can be realized as the representation category $\textrm{Rep}(\mcV)$ of a VOA $\mcV$ \cite{TW17}.  This conjecture is obviously true for quantum group categories $\mathfrak{G}_k$ as they are realized by the WZW models $\mathfrak{G}_k$.  Another large class is the quantum double of finite groups \cite{EG18}.  Systematical investigation for low rank anyon models also supports the conjecture \cite{Junla14, TW17}.  The challenging anyon models are those arising as Drinfeld centers in subfactor theories such as the famous doubled Haagerup.  The reconstruction conjecture would imply they could also also realized by VOAs in the same way that the toric code can be realized by $SO(16)_1$.  Further evidence is provided by abelian anyon models, which all can be reconstructed as representation categories of lattice VOAs \cite{Nik1, EG20}.   The first existence proof is not constructive and a new proof can be made more constructive. Our main technical tools include an adaption of an algorithm by Wall in \cite{Wall63} and a complement lattice lemma in \cite{EG20}.

Our motivation for the reconstruction is the bulk-edge correspondence of topological phases of matter.  The bulk-edge correspondence for abelian anyon models has been extensively investigated in \cite{Wen92, WZ92,Cano14} with a focus on boundary lattice VOAs.  Instead our interests lie in stability of the edge theory, equivariant correspondence, and non-lattice VOA boundaries \cite{Haldane95}.  The stability leads us to focus on extremal $\chi$CFTs with minimal central charge realization, hence a different emphasis from \cite{Nik1, EG20}: \cite{Nik1} focuses on abstract existence, while \cite{EG20} is a detailed reconstruction through lattice gluing.  

The genus of a unitary $\chi$CFT $\mcV$ is the pair $(\textrm{Rep}(\mcV),c)$, where $\textrm{Rep}\mcV$ is an anyon model and $c$ is the central charge. By \cite{Nik1}, the minimal genus of each abelian anyon model can always be realized.  While we provided many explicit reconstructions in the minimal genera, we  did not succeed in realizing every one.

Along the way, we can deduce some interesting consequences about abelian anyon models.  For example, the central charge of a cyclic anyon model $A$ is always of the opposite parity of the order of $A=p^r$ for some prime $p$. The number of prime factors of an abelian anyon model is not an invariant as the product of the semion theory with toric code is the same as the product with doubled semions.  Finally eight copies of any abelian anyon model is always a Drinfeld center, which also follows from the structure of the pointed part of the Witt group \cite{DMNO10}.

The content of the paper is as follows. In Sec. \ref{summary}, we summarize our results and illustrate our work using some examples.  In Sec. \ref{classification}, we list all prime abelian anyon models and calculate their central charges and symmetry groups.  In Sec. \ref{conkmatrix}, we construct a $K$-matrix for each prime abelian anyon model explicitly. In Sec. \ref{conlattice}., we construct a lattice VOA realization for each prime abelian anyon model.  Finally we made some conjectures in and around the stability, equivariant bulk-edge correspondence, and non-lattice boundary VOAs.

\subsection{Notations}

\begin{enumerate}

\item All lattices $\Lambda$ in this paper are non-degenerate and integral over $\mbbZ$ except their dual lattices $\Lambda^*$ and those explicitly stated otherwise.

\item A $K$-matrix $K_A$ associated to an abelian anyon model $(A,q)$ is an even non-degenerate integral symmetric matrix.  When $K_A$ is positive-definite (or $+$-definite), it is also a Gram matrix for a lattice $L_A$.

\item A metric group is a pair $(A,\theta)$, where  $A$ is a finite abelian group, and $\theta: A\rightarrow \mathbb{U}(1)$ is a non-singular quadratic form in the sense that the associated bi-character $\chi(x,y)=\frac{\theta(x+y)}{\theta(x)\theta(y)}$ is bi-multiplicative and non-degenerate. 

A metric group $(A,\theta)$ will be also denoted as $(A,q)$ for a quadratic form $q: A\rightarrow \mbbQ/\mbbZ$ such that $\theta(x)=e^{2\pi i q(x)}$.

\item An abelian anyon model is a unitary modular tensor category that realizes the metric group $(A,\theta)$, so the two terms will be used interchangeably.

\item  Given a lattice $L$, its discriminant group is $A_L=L^*/L$.  The generating matrix of $L$ will be denoted by $G_L$ and the Gram matrix by $K_L$.

The discriminant form, the quadratic form on the finite discriminant group $A$ for a non-degenerate even lattice, is $q_{2,A}: A \rightarrow \mathbb{Q}/2\mathbb{Z}$, i.e. $q_{2,A}(-x)=q_{2,A}(x)$ and the associated bilinear map $b(x,y)=\frac{1}{2}(q_{2,A}(x+y)-q_{2,A}(x)-q_{2,A}(y)): A\times A \rightarrow \mathbb{Q}/\mathbb{Z}$ is symmetric and non-singular, where $b(,)$ is the extension of the bilinear $<,>$ on lattice $\Lambda$ to its dual.

\end{enumerate}

\section{Summary and Examples}\label{summary}

In this section, we summarize our results and illustrate our work using some examples.  The classification of abelian anyon models can be derived from the literature from various resources \cite{JS93,Durfee,Wall63, Galindo16}.  We will follow the notation and work of \cite{Wall63} to list all primes ones Thm. \ref{classification_abelian} and note that there are relations among them that are also completely known \cite{Miranda84, KK80}.  The relations among various models illustrate a known phenomenon that the prime decomposition of abelian anyon models is not unique \cite{Mu03}, even the number of prime factors is not an invariant.

Our first contribution is to calculate the central charge and symmetry group  for each prime abelian anyon model in Thms. \ref{ calculate c}, \ref{Aut}.  The calculation of central charge can also be derived from the known literature, e.g. \cite{EG20}.  The symmetry group is new, though in one case we cannot decide the abstract group.

Our main interest is an algorithmic construction of $K$-matrices and positive definite ones for each prime anyon model with the smallest posible central charge.  Their existence is already known by \cite{Nik1}, so the new ingredient is an explicit construction, which the second author has been asked by several physicists over the years.
One of the interesting future directions is to generalize the Monster moonshine module from the trivial anyon model to all prime abelian anyon models.

We divide prime abelian anyon models into $8$ families in Thm.  \ref{classification_abelian}.  Families always come in pairs as each anyon model has a conjugate, which could the same. The $8$ families are: two for an odd prime $p$, and six for the prime $2$.  For an odd prime $p$, the rank is a prime power $p^r$ for $r\geq 1$.  When $r$ is even, then the abelian anyon model is always a quantum double \cite{Ardonne16}.  When $r$ is odd, the models naturally divide into four cases by the residue of $p$ mod $8$ with central charges $0,2,4,6$ mod $8$, respectively, and their conjugates.  For the prime $2$, the first four families are two pairs with central charges $1,3,5,7$ mod $8$.  There are two additional families with central charges $0,4$ mod $8$, which are more interesting: the $E$ family is a generalization of the famous toric code, while the $F$ family a generalization of the three fermion model.

For CS-theory, we need to find a $K$-matrix $K_A$ for an abelian anyon model $A$.  Our main tool is an adaption of an algorithm in \cite{Wall63}, which we refer to as the Wall algorithm.  For a chiral CFT realization, we need $K_A$ to be $+$-definite, which is significantly harder.

From each of the $8$ families, we pick one example to illustrate our results.
The realization of abelian anyon models by $+$-definite $K$-matrices can be regarded as certain generalization of Cartan ADE matrices.  
The ADE lattices are the most well-known $+$-definite $K$-matrices and realize the abelian anyon models $(A,q_{2,L})$ given by rows two and three in the following table:

\begin{tabular}{|c|c|c|c|c|c|c|}

\hline

\textrm{Lattice}&$A_{r-1}$&$D_{2s}$&$D_{2s+1}$&$E_6$&$E_7$&$E_8$\\
\cline{3-6}
\hline

\textrm{Discriminant group $A$}&$\mbbZ/r\mbbZ$&$\mbbZ_2\times \mbbZ_2$&$\mbbZ/4\mbbZ$&$\mbbZ/3\mbbZ$&$\mbbZ/2\mbbZ$&$1$\\
\cline{3-6}
\hline

\textrm{Discriminant form $q_{2,L}$}&$\frac{r-1}{r}$&$\{0,s/2,s/2,(s-4)/2\}$&$\frac{2s+1}{4}$&$\frac{4}{3}$&$\frac{3}{2}$&$0$\\
\hline
\end{tabular}

For the odd prime families, we will focus on $p=3$ and $p=5$ because they are related to the statistics sectors of FQH liquids at $\nu=\frac{1}{3}$ and $\nu=\frac{1}{5}$, respectively.  While the edge physics of $\nu=\frac{1}{3}$ is well-studied, our analysis for $p=5$ suggests there might be interesting edge physics at $\nu=\frac{1}{5}$, even possible new edge physics for the  integer quantum Hall and $\nu=\frac{1}{3}$.

For $p=3$, the anyon model denoted as $\mathbb{Z}_3$ has three anyon types with central charge=$2$ and $6$.  The realization of central charge =$2$ is the WZW model $SU(3)_1$.  The $K$ matrix is 
$\begin{pmatrix}
2 &1\\
1&2
\end{pmatrix}$, which is $+$-definite already.  The central charge=$6$ realization is the Cartan matrix $E_6$, which is also $+$-definite.

For $p=5$, the anyon model denoted as $\mathbb{Z}_5$ has five anyon types with central charge=$0$ and $4$ mod $8$.  The realization of central charge=$0$ mod $8$ will be the minimal central charge=$8$.  The solution for $c=8$ is unique given by the following $K$-matrix obtained from the classification of lattices \cite{CNbook}, while our algorithm produces a solution with $c=32$:

$$\begin{pmatrix}

2&-1& 0 &0&0&0&0&0\\
-1&2&1 &0&0&0&0&0\\
0&-1&2&-1 &0&0&0&0\\
0&0&-1&2&-1&0&-1&0\\
0&0&0&-1&2&-1&0&0\\
0&0&0&0&-1&2&0&-1\\
0&0&0&-1&0&0&2&1\\
0&0&0&0&0&-1&1&4

\end{pmatrix} .$$

For the central charge=$4$ mod $8$, our algorithm produces a minimal solution: 
$$\begin{pmatrix}
20 & -15 & 10 &-5\\
-15& 12 & -8 & 4\\
10 & -8 &6 & -3\\
-5 &4 & -3 & 2
\end{pmatrix} ,$$ which is $+$-definite already.

For the prime $2$, the anyon model denoted as $\mathbb{Z}_2$ has two anyon types with central charge=$1$ and $7$ mod $8$. They can be realized by WZW models $SU(2)_1$ and $(E_7)_1$, the semion and anti-semion pair.

The anyon model denoted as $\mathbb{Z}_4$ has $4$ anyon types, and there are two pairs.  The first pair can be realized by the $K$ matrices $(4)$ and Cartan matrix $D_7$.  The another pair can be realized by WZW models $SU(4)_1$ and $(D_5)_1$.

The toric code can be realized by WZW model $(D_8)_1$.  The next one in the $E$ family has $\mathbb{Z}_4\times \mathbb{Z}_4$ as fusion rules.  By our algorithm, to construct a positive definite $K$-matrix, we start with one of $B_4$, which can be chosen to be the $D_7$ as above.  Then we obtain the following  $16\times 16$ $+$-definite solution.  We do not know if there is an $8\times 8$ solution.

$$
\begin{pmatrix}
4 & 1& -1& 0& 0& 0& 0& 0& 0& -1& 0& 0& 0& 0& 0&0\\
1 & 4& 0& 0& 0& 0 & 0 & 0& 0& 0& 0& 0& 0 &0 & 0& 0\\
-1 &0 &2& 0& -1& 0 &0 &0 &0 &0 &0& 0 &0& 0& 0& 0\\
0 &0& 0& 2 &-1& 0 &0& 0 &0 &0 &0 &0 &0& 0 &0& 0\\
0 &0 &-1& -1& 2 &-1& 0& 0& 0& 0 &0 &0& 0 &0& 0 &0\\
0& 0 &0 &0 &-1& 2& -1 &0& 0& 0& 0& 0 &0 &0 &0 &0\\
0 &0& 0& 0 &0& -1& 2 &-1& 0& 0 &0 &0 &0 &0 &0 &0\\
0 &0& 0& 0& 0& 0 &-1& 2& -1& 0 &0 &0 &0 &0 &0 &0\\
0 &0 &0 &0& 0& 0& 0 &-1 &2& 0& 0 &0& 0& 0 &0& 0\\
-1 &0& 0 &0 &0 &0& 0 &0 &0 &2 &0 &-1 &0 &0 &0 &0\\
0 &0 &0 &0 &0& 0& 0& 0& 0& 0 &2 &-1 &0& 0 &0& 0\\
0 &0 &0& 0 &0 &0& 0 &0 &0 &-1& -1 &2 &-1& 0 &0 &0\\
0& 0& 0 &0& 0 &0 &0 &0 &0 &0 &0 &-1 &2 &-1 &0 &0\\
0& 0& 0& 0 &0& 0 &0 &0 &0 &0 &0& 0 &-1& 2 &-1 &0\\
0 & 0& 0& 0& 0 &0 &0 &0 &0 &0 &0 &0 &0 &-1 &2 &-1\\
0 &0 &0& 0 &0& 0 &0 &0 &0 &0 &0& 0 &0 &0 &-1 &2
\end{pmatrix}.
$$

The three fermion model denoted as $F_{\mathbb{Z}_2 \times \mathbb{Z}_2}$ is realized by $(D_4)_1$.  The next one in the family with fusion rules $\mathbb{Z}_2^2 \times \mathbb{Z}_2^2$ can be realized using the following $4\times 4$ $+$-definite $K$-matrix by our algorithm:

$$\begin{pmatrix}
8&20&-8&-4\\
20&-40&16&8\\
-8&16&-6&-3\\
-4&8&-3&-2 
\end{pmatrix}.
$$

For topological symmetries, for odd prime and the first $4$ families of prime=$2$, the only possible faithful symmetry is the charge conjugation $\mathbb{Z}_2$.  The $E$ families and $F$ famileis have very interesting large faithful symmetries.  For example, the generalized three fermion model with fusion rules $\mathbb{Z}_2^2 \times \mathbb{Z}_2^2$ has the topological 
symmetry group the dihedral group with $12$ elements $D_6$. It would be very interesting to explore the large symmetry groups in gauging and physical applications.

The enumeration, central charges, and relations among abelian anyon models can be used to deduce and realize the structure theorem of the pointed part of the Witt group \cite{DNO13}.  The pointed part of the Witt group decomposes into direct sum over primes.  The prime=$2$ part is $\mbbZ_8\oplus \mbbZ_2$, where the generators are a $\mbbZ_4$ model and the three-fermion.  When an odd prime $p=3$ mod $4$,  the $\mbbZ_4$ is generated by one with central charge $c=2$, while $p=1$ mod $4$, the $\mbbZ_2\oplus \mbbZ_2$ is generated by $c=4$ models.

\section{Classification of abelian anyon models and their central charges and topological symmetry groups}\label{classification}

In this section, we list the classification of prime abelian anyon models and calculate their central charges and topological symmetry groups.  While the classification and central charges can be deduced from the literature, the calculation of topological symmetry groups seems to be new.  The calculation of central charges is useful for identifying conjugate pairs of prime abelian anyon models and studying the pointed part of the Witt group of modular categories.

Abelian anyon models play an important role in the study of topological order and topological phases of matter.  As algebraic data for abelian quantum Chern-Simons theories, they are realized by abelian fractional quantum Hall states. 

Mathematically, abelian anyon models $\mcA$ are pointed unitary modular tensor categories.  There is a one-one correspondence between metric groups $(A,q)$ and abelian anyon models $\mcA$\footnote{The correspondence between metric groups and pointed modular tensor category is not one-to-one without unitarity.  There are $8$ premodular categories of fusion rule $\mathbb{Z}_2$: four are modular, and four unitary.  Among them, there are only two unitary modular categories: the semion and anti-semion theories.  There is a non-unitary modular category with the same $T$ matrix as semion: Temperley-Lieb theory with Kauffman variable $A=e^{\pm \frac{2\pi i}{12}}$ for $r=3$. }.
Thus, the classification of abelian anyon models is reduced to the classification of metric groups, which can be deduced from Wall's classification of unimodular symmetric bilinear forms $b(\cdot,\cdot): L_A\times L_A \rightarrow \mathbb{Q}/\mathbb{Z}$ on free abelian groups $L_A$ and many other places \cite{Wall63, Durfee}.

\subsection{Prime decomposition of abelian anyon models}

In an abelian anyon model $\mcA$, the anyon types form a finite abelian group $A$ under fusion or tensor product.  The topological twist 
$\theta: A\rightarrow  \mbbU(1)$ of each anyon type is a non-degenerate quadratic form on $A$.  Besides $\theta$, we will also use the notation $(A,q)$ to denote a metric group, where $q: A\rightarrow {\mbbQ}/{\mbbZ}$ is a quadratic form such that $\theta(a)=e^{2\pi i q(a)}$.  

Every modular tensor category decomposes as a Deligne product $\boxtimes$ of prime ones \cite{Mu03}.  Prime modular categories are those that are not any non-trivial products.  Such a prime decomposition is not always unique.
A general family of counterexamples are the Drinfeld centers of different anyon models with the same underlying unitary fusion categories because the Drinfeld center is independent of the braidings, e.g. two copies of the toric code is the same as two copies of the three-fermion. Moreover, the number of prime factors in prime  decompositions can be different as the example $D\mbbZ_2 \boxtimes \textrm{Semion}=D\textrm{Semion}\boxtimes \textrm{Semion}$ shows.

The Deligne product $\boxtimes$ of abelian anyon models corresponds to direct sum $\oplus$ of metric groups.  Therefore, a classification of abelian anyon models is the same as a classification of finite groups with non-degenerate quadratic forms and the relations among them.  Our list follows from Wall's classification of quadratic forms on finite abelian groups.  Wall also found all relations among the generators for the odd prime case.  The relations for the prime=$2$ case is much more complicated, but have been worked out in \cite{Miranda84,KK80}.

\subsection{Classification of prime abelian anyon models}

There are eight families of prime abelian anyon models with known nontrivial relations among them.  We follow Wall's notation.

\begin{thm}\label{classification_abelian}
There are eight families of prime abelian anyon models $(A,\theta)$ whose topological twists $\theta(a)=e^{2\pi i q(a)}, a\in A$ are given by one of the following maps $q(a): A\rightarrow \mbbQ/\mbbZ$ on the finite abelian group $A$. 

(a): Let $p$ be an odd prime and $1$ the generator of the cyclic group $A=\mbbZ_{p^r}$: 
\begin{align}
A_{p^r} \ :\ q(1) = \frac{m}{p^r}, \ \text{for some }\ 1\leq m < p, \ (m,p)=1  \ \text{and}  \ \Big (\frac{2m}{p}\Big )=1.\\
B_{p^r} \ :\ q(1) = \frac{n}{p^r}, \ \text{for some }\ 1\leq n< p, \ (n,p)=1  \ \text{and} \ \Big (\frac{2n}{p}\Big )=-1,
\end{align}
where $\Big (\frac{x}{p}\Big )$ is the Legendre symbol.

The different choices of $m,n$ in each family lead to the same theory.  Theories come in conjugate pairs and the two families for $p=-1$ mod $4$ are conjugates of each other.

(b): There are six families for the prime=$2$.  The first 
four of which are for the cyclic group $A=\mbbZ_{2^r}$, where for $A_{2^r}$ and $B_{2^r},r\geq1$, for $C_{2^r}$ and $D_{2^r},r\geq2$:
\begin{align}
A_{2^r}\ :\  q(1)=\frac{1}{2^{r+1}} ,\\
B_{2^r}\ :\  q(1)=-\frac{1}{2^{r+1}} ,\\ 
C_{2^r}\ :\  q(1) = \frac{5}{2^{r+1}} ,\\ 
D_{2^r}\ :\  q(1) = \frac{-5}{2^{r+1}} . 
\end{align}

These four families consist of two conjugage pairs.  For $r=1$, they are the Semion and anti-Semion pair.

Two additional families are for the abelian groups $A=\mbbZ_{2^r} \times \mbbZ_{2^r}$ with basis $e_1=(1,0),e_2=(0,1)$. Denote $e_\alpha=me_1+ne_2$ for $\alpha=(m,n) \in \mbbZ^2, 0\leq m,n\leq 2^r-1$, then:

\begin{align}
E_{2^r}\ :\ q(e_{(m,n)})=\frac{mn}{2^r},
\end{align}
\begin{align}
F_{2^r}\ : \ q(e_{(m,n)})=\frac{m^2+n^2+mn}{2^r}.
\end{align}

For $r=1$, the $E_{2^r}$ model is the toric code, while $F_{2^r}$ the three-fermion theory.

\end{thm}

We will refer to them as families $A,B,C,D,E,F$ in the future.

Among the odd prime=$p$ abelian anyon models, the only nontrivial relations are $2A_{p^r}=2B_{p^r}$.  For prime $p=2$, there are many nontrivial relations which are determined and completely listed in \cite{Miranda84, KK80}.  Of particular interest is the fact that four copies of any theory with the same group $A$ give the same theory for all theories for $p=2$.  The braidings and $6j$-symbols of all the anyon anyon models are given by third abelian cohomology, which can be found e.g. in \cite{Galindo16}.

\subsection{Central charges}

Given any modular tensor category with label set $\Pi$, there is a rational number $c$ which is defined modulo $8$ by the equation:

$$\frac{\sum_{a\in \Pi}\theta_a d_a^2}{\sqrt{\sum_{a\in \Pi} d_a^2}}=e^{\frac{\pi i c}{4}}.$$
For abelian anyon models, all $d_a=1$, hence the central charge $c$ is determined by 
$$ \frac{\sum_{a\in A}\theta_a}{\sqrt{|A|}}=e^{\frac{\pi i c}{4}}.$$
Since the central charge is additive under Deligne product, one consequence of our calculation is that the central charge of an abelian anyon model is always an integer, and with the opposite parity of the order of $A$ in the families $A,B,C,D$.

The central charges of all eight families of prime abelian anyon models in Thm. \ref{classification_abelian} are calculated and listed below.  They can also be deduced from the literature, e.g. \cite{EG20}.  Since the central charge is additive with respect to the direct sum decomposition, therefore the central charge of any abelian anyon model follows.

\begin{thm}\label{ calculate c}

The central charges of the eight families of prime abelian anyon models are as in Table \ref{cctable}.

\begin{table}[!htbp]\label{cctable}
  \centering 
 \caption{The relations of $p, r,c$}  
 
\begin{tabular}{|c|c|c|c|c|c|}

\hline
\multirow{2}*{}& \multirow{2}*{r even } &\multicolumn{4}{|c|}{r odd}\\
\cline{3-6}

{}&{}&$p\equiv1 $\;mod\;8&$p\equiv-1 $\;mod\;8&$p\equiv-3 $\;mod\;8&$p\equiv3 $\;mod\;8\\
\cline{3-6}
\hline

$A_{p^r}$&0&0&2&4&6\\
\cline{3-6}
\hline

$B_{p^r}$&0&4&6&0&2\\
\hline
\end{tabular}

\begin{tabular}{|c|c|}

\hline  
{}&all r\\
\hline  
$A_{2^r}$&1\\
\hline
$B_{2^r}$&7\\
\hline
$E_{2^r}$&0\\
\hline 
\end{tabular}

\begin{tabular}{|c|c|c|}
\hline  
{}& r even&r odd\\
\hline  
$C_{2^r}$&5&1\\
\hline
$D_{2^r}$&3&7\\
\hline
$F_{2^r}$&0&4\\
\hline 
\end{tabular}

\end{table}

More concretely, if $p$ is an odd prime and $1$ the generator of $\mbbZ_{p^r}$, then:
\begin{align}
A_{p^r} \ :\ q(1) = \frac{m}{p^r}, \ \text{where } (m,p)=1\ \text{and} \ \Big (\frac{2m}{p}\Big )=1,\\
c\equiv\begin{cases}0&\text{r even };\text{or r odd }p\equiv1\; mod\;8\\2&\text{r odd }p\equiv-1 \;mod\;8\\6&\text{r odd }p\equiv3\; mod\;8\\4&\text{r odd }p\equiv-3 \;mod\;8\\\end{cases}.\nonumber\\
B_{p^r} \ :\ q(1) = \frac{n}{p^r}, \ \text{where } (n,p)=1\ \text{and} \ \Big (\frac{2n}{p}\Big )=-1, \\
c\equiv\begin{cases}0&\text{r even };\text{or r odd }p\equiv-3\; mod\;8\\4&\text{r odd }p\equiv1\; mod\;8\\6&\text{r odd }p\equiv-1\; mod\;8\\2&\text{r odd }p\equiv3\; mod\;8\\\end{cases}\nonumber \\
\end{align}

For the other six families, four of which for $\mbbZ_{2^r}$ are listed as below:
\begin{align}
A_{2^r}\ :\  q(1)=\frac{1}{2^{r+1}}, c\equiv 1 \; mod \;8.\\
B_{2^r}\ :\  q(1)=-\frac{1}{2^{r+1}}, c\equiv 7\; mod\; 8.\\
C_{2^r}\ :\  q(1) = \frac{5}{2^{r+1}},c\equiv \begin{cases} 1&\text{ r odd }\\ 5&\text{ r even}
\end{cases}\;mod\; 8. \\
D_{2^r}\ :\  q(1) = \frac{-5}{2^{r+1}},c\equiv \begin{cases} 7&\text{ r odd } \\ 3&\text{ r even}
\end{cases}\;mod\; 8.
\end{align}
Note that for $A_{2^r}$ and $B_{2^r},r\geq1$, for $C_{2^r}$ and $D_{2^r},r\geq2$.

The last two families have abelian groups $\mbbZ_{2^r} \times \mbbZ_{2^r}$ with basis $e_1=(1,0),e_2=(0,1)$. Denote $e_\alpha=me_1+ne_2$ for $\alpha=(m,n) \in \mbbZ^2$.
\begin{align}
E_{2^r}\ :\ q(e_{(m,n)})=\frac{mn}{2^r},  c\equiv 0\; mod\; 8.\\
F_{2^r}\ : \ q(e_{(m,n)})=\frac{m^2+n^2+mn}{2^r},\;
 c\equiv \begin{cases} 4&\text{ r odd }\\ 0&\text{ r even}
\end{cases}\;mod\;8. \nonumber\\
\end{align}

\end{thm}

\begin{proof}
First a well-known result for Gauss sums  $\mathcal{G}(n)=\sum\limits_{j=1}^{n}e^{\frac{2\pi ij^2}{n}}$ is: \[\mathcal{G}(n)=\begin{cases} \sqrt{n},&n\equiv1 \;mod\;4\\ 0,&n\equiv2 \;mod\;4\\ i\sqrt{n},&n\equiv3\; mod\;4\\(1+i)\sqrt{n},&n\equiv0 \;mod\;4\\
\end{cases}\]
and for $\mathcal{G}(n,m)=\sum\limits_{j=1}^{n}e^{\frac{2\pi imj^2}{n}}$, where $n$ is odd and $(n,m)=1$, we have \[\mathcal{G}(n,m)=(\frac{m}{n})\mathcal{G}(n,1)=
(\frac{m}{n})\mathcal{G}(n),\] where $(\frac{m}{n})$ is the Jacobi symbol.

Similar to quadratic reciprocity, the following holds:
\[\sum\limits_{n=1}^{c}e^{\frac{\pi i an^2}{c}}=\sqrt{\frac{c}{a}
}e^{\frac{\pi i}{4}}\sum\limits_{n=1}^{a}e^{-\frac{\pi i cn^2}{a}},\]
where $a$ and $c$ are positive integers and $ac$ is even.

\begin{itemize}
  \item For $A_{p^r} \ :\ q(1) = \frac{m}{p^r}, \ \text{where } (m,p)=1\ \text{and} \ \Big (\frac{2m}{p}\Big )=1,\\
e^{\frac{\pi ic}{4}}=\frac{\sum\limits_{j=1}^{p^r}e^{2\pi iq(j)}}{p^{\frac{r}{2}}}=\frac{\sum\limits_{j=1}^{p^r}e^{\frac{2\pi imj^2}{p^r}}}{p^{\frac{r}{2}}}=(\frac{m}{p^r})
\frac{\sum\limits_{j=1}^{p^r}e^{\frac{2\pi ij^2}{p^r}}}{p^{\frac{r}{2}}}\\
=(\frac{m}{p})^r
\frac{\sum\limits_{j=1}^{p^r}e^{\frac{2\pi ij^2}{p^r}}}{p^{\frac{r}{2}}}
=(\frac{2}{p})^r
\frac{\sum\limits_{j=1}^{p^r}e^{\frac{2\pi ij^2}{p^r}}}{p^{\frac{r}{2}}}=\begin{cases}1&\text{r even };\text{or r odd }p\equiv1 \;mod\;8\\i&\text{r odd }p\equiv-1 \;mod\;8\\-i&\text{r odd }p\equiv3 \;mod\;8\\-1&\text{r odd }p\equiv-3 \;mod\;8.\\\end{cases}$

Thus
\[c\equiv\begin{cases}0&\text{r even };\text{or r odd }p\equiv1 \;mod\;8\\2&\text{r odd }p\equiv-1 \;mod\;8\\6&\text{r odd }p\equiv3 \;mod\;8\\4&\text{r odd }p\equiv-3\; mod\;8.\\\end{cases}
\]  

  \item For  $B_{p^r} \ :\ q(1) = \frac{n}{p^r}, \ \text{where } (n,p)=1\ \text{and} \ \Big (\frac{2n}{p}\Big )=-1, \\
    e^{\frac{\pi ic}{4}}=\frac{\sum\limits_{j=1}^{p^r}e^{2\pi iq(j)}}{p^{\frac{r}{2}}}=\frac{\sum\limits_{j=1}^{p^r}e^{\frac{2\pi inj^2}{p^r}}}{p^{\frac{r}{2}}}=(\frac{n}{p^r})
\frac{\sum\limits_{j=1}^{p^r}e^{\frac{2\pi ij^2}{p^r}}}{p^{\frac{r}{2}}}\\
=(\frac{n}{p})^r
\frac{\sum\limits_{j=1}^{p^r}e^{\frac{2\pi ij^2}{p^r}}}{p^{\frac{r}{2}}}
=(-1)^r(\frac{2}{p})^r
\frac{\sum\limits_{j=1}^{p^r}e^{\frac{2\pi ij^2}{p^r}}}{p^{\frac{r}{2}}}=\begin{cases}1&\text{r even};\text{or r odd }p\equiv-3\; mod\;8\\-1&\text{r odd }p\equiv1 \;mod\;8\\-i&\text{r odd }p\equiv-1 \;mod\;8\\i&\text{r odd }p\equiv3 \;mod\;8.\\\end{cases}$

Thus
\[c\equiv\begin{cases}0&\text{r even };\text{or r odd }p\equiv-3 \;mod\;8\\4&\text{r odd }p\equiv1 \;mod\;8\\6&\text{r odd }p\equiv-1 \;mod\;8\\2&\text{r odd }p\equiv3 \;mod\;8.\\\end{cases}
\]

\item For $A_{2^r}\ :\  q(1)=\frac{1}{2^{r+1}},e^{\frac{\pi ic}{4}}=\frac{\sum\limits_{j=1}^{2^r}e^{2\pi iq(j)}}{2^{\frac{r}{2}}}=\frac{\sum\limits_{j=1}^{2^r}e^{\frac{2\pi ij^2}{2^{r+1}}}}{2^{\frac{r}{2}}}
=\frac{\sum\limits_{j=1}^{2^r}e^{\frac{\pi ij^2}{2^r}}}{2^{\frac{r}{2}}}
=\frac{2^{\frac{r}{2}}e^{\frac{\pi i}{4}}}{2^{\frac{r}{2}}}=e^{\frac{\pi i}{4}}$.

For the second last equality, quadratic reciprocity is used 
for $c=2^r,a=1$.
It follows that $c\equiv1 \;mod\;8.$

  \item For $B_{2^r}\ :\  q(1)=-\frac{1}{2^{r+1}},
e^{\frac{\pi ic}{4}}=\frac{\sum\limits_{j=1}^{2^r}e^{2\pi iq(j)}}{2^{\frac{r}{2}}}=\frac{\sum\limits_{j=1}^{2^r}e^{-\frac{2\pi ij^2}{2^{r+1}}}}{2^{\frac{r}{2}}}
=\overline{\frac{\sum\limits_{j=1}^{2^r}e^{\frac{2\pi ij^2}{2^{r+1}}}}{2^{\frac{r}{2}}}}=e^{-\frac{\pi i}{4}}$.

Therefore, $c\equiv-1 mod8.$

\item For $C_{2^r}\ :\  q(1) = \frac{5}{2^{r+1}},e^{\frac{\pi ic}{4}}=\frac{\sum\limits_{j=1}^{2^r}e^{2\pi iq(j)}}{2^{\frac{r}{2}}}=\frac{\sum\limits_{j=1}^{2^r}e^{\frac{2\pi i5j^2}{2^{r+1}}}}{2^{\frac{r}{2}}}
=\frac{\sum\limits_{j=1}^{2^r}e^{\frac{\pi i5j^2}{2^r}}}{2^{\frac{r}{2}}}
=\sqrt{\frac{1}{5}}e^{\frac{\pi i}{4}}
\sum\limits_{j=1}^{5}e^{-\frac{\pi i 2^rj^2}{5}}\\
=\sqrt{\frac{1}{5}}e^{\frac{\pi i}{4}}
\overline{(\frac{2^{r-1}}{5})\sum\limits_{j=1}^{5}e^{\frac{2\pi i j^2}{5}}}
=\sqrt{\frac{1}{5}}e^{\frac{\pi i}{4}}
\overline{(-1)^{r-1}\sqrt{5}}=e^{\frac{\pi i}{4}}(-1)^{r-1}$.

For the fourth equality, quadratic reciprocity is used 
for $c=2^r,a=5$.
For the second last equality, $(\frac{2}{5})=-1$ and $\mathcal{G}(5)=\sqrt{5}$ are used.

Hence
\[c\equiv\begin{cases}1& \text{r odd}\\5& \text{r even }\end{cases}\;mod\;8.\]

\item For $D_{2^r}\ :\  q(1) = \frac{-5}{2^{r+1}},
e^{\frac{\pi ic}{4}}=\frac{\sum\limits_{j=1}^{2^r}e^{2\pi iq(j)}}{2^{\frac{r}{2}}}=\frac{\sum\limits_{j=1}^{2^r}e^{\frac{-2\pi i5j^2}{2^{r+1}}}}{2^{\frac{r}{2}}}
=\overline{\frac{\sum\limits_{j=1}^{2^r}e^{\frac{2\pi i5j^2}{2^{r+1}}}}{2^{\frac{r}{2}}}}
=\overline{e^{\frac{\pi i}{4}}(-1)^{r-1}}
=e^{\frac{-\pi i}{4}}(-1)^{r-1}$

Thus
\[c\equiv\begin{cases}-1& \text{r odd}\\-5& \text{r even }\end{cases}\;mod\;8.\]

\item For $E_{2^r}\ :\ q(e_{(m,n)})=\frac{mn}{2^r},
e^{\frac{\pi ic}{4}}=\frac{\sum\limits_{\alpha\in \mbbZ_{2^r} \times \mbbZ_{2^r} }e^{2\pi iq(e^{\alpha})}}{2^r}
=\frac{\sum\limits_{m=1}^{2^r}\sum\limits_{n=1}^{2^r}e^{\frac{2\pi imn}{2^r}}}{2^r}$.

Set $f(r)=\frac{\sum\limits_{m=1}^{2^r}\sum\limits_{n=1}^{2^r}e^{\frac{2\pi imn}{2^r}}}{2^r}$, an easy calculation shows $f(1)=f(2)=1.$ When $r\geq3,$
\begin{align*}
  f(r)& =\frac{\sum\limits_{m=1}^{2^r}\sum\limits_{n=1}^{2^r}e^{\frac{2\pi imn}{2^r}}}{2^r}\\
    &=\frac{\sum\limits_{m=1,m \,odd.}^{2^r}\sum\limits_{n=1}^{2^r}e^{\frac{2\pi imn}{2^r}}}{2^r}
+\frac{\sum\limits_{m=1,m \, even.}^{2^r}\sum\limits_{n=1}^{2^r}e^{\frac{2\pi imn}{2^r}}}{2^r}\\
  &=\frac{\sum\limits_{m=1,m \,odd.}^{2^r}\sum\limits_{l=1,l=mn.}^{2^r}e^{\frac{2\pi il}{2^r}}}{2^r}
+\frac{\sum\limits_{m=1,m \, even.}^{2^r}\sum\limits_{n=1}^{2^r}e^{\frac{2\pi imn}{2^r}}}{2^r}\\
&=\frac{\sum\limits_{m=1,m \, even.}^{2^r}\sum\limits_{n=1}^{2^r}e^{\frac{2\pi imn}{2^r}}}{2^r}\\
& =\frac{\sum\limits_{s=1,s=\frac{m}{2} .}^{2^{r-1}}\sum\limits_{n=1}^{2^r}e^{\frac{2\pi i2sn}{2^r}}}{2^r}\\
\end{align*}

For $\sum\limits_{n=1}^{2^r}e^{\frac{2\pi imn}{2^r}}=
\sum\limits_{l=1,l=mn.}^{2^r}e^{\frac{2\pi il}{2^r}}$, we use the fact that when $m$ runs a complete residue system modulo $2^r$, then
$l=mn$ runs a complete residue system modulo $2^r$ if $m$ is odd. Suppose $1\leq s,t\leq 2^r$, and $ms\equiv mt$ mod $2^r$, then $2^r | m(s-t)$. As $m$ is odd, thus $2^r| (s-t)$. But $1\leq s,t\leq 2^r $, then $s=t$, thus claim verified!
Next
\begin{align*}
f(r)&=\frac{\sum\limits_{n=1}^{2^r}\sum\limits_{s=1}^{2^{r-1}}e^{\frac{2\pi isn}{2^{r-1}}}}{2^r}\\
&=\frac{\sum\limits_{n=1}^{2^{r-1}}\sum\limits_{s=1}^{2^{r-1}}e^{\frac{2\pi isn}{2^{r-1}}}}{2^r}+
\frac{\sum\limits_{t=1,t=n-2^{r-1}.}^{2^{r-1}}\sum\limits_{s=1}^{2^{r-1}}e^{\frac{2\pi is(t+2^{r-1})}{2^{r-1}}}}{2^r}\\
&=\frac{f(r-1)}{2}+
\frac{\sum\limits_{t=1}^{2^{r-1}}\sum\limits_{s=1}^{2^{r-1}}e^{\frac{2\pi ist}{2^{r-1}}}}{2^r}\\
&=\frac{f(r-1)}{2}+\frac{f(r-1)}{2}\\
&=f(r-1)\\
\end{align*}
By induction, $f(r)=1$ for all $r\geq1.$
Then $e^{\frac{\pi ic}{4}}=f(r)=1.$
It follows that $c\equiv0 \;mod\;8.$

\item For $F_{2^r}\ : \ q(e_{(m,n)})=\frac{m^2+n^2+mn}{2^r}.
e^{\frac{\pi ic}{4}}=\frac{\sum\limits_{\alpha\in \mbbZ_{2^r} \times \mbbZ_{2^r} }e^{2\pi iq(e^{\alpha})}}{2^r}
=\frac{\sum\limits_{m=1}^{2^r}\sum\limits_{n=1}^{2^r}e^{\frac{2\pi i(m^2+n^2+mn)}{2^r}}}{2^r}$.

Set $g(r)=\frac{\sum\limits_{m=1}^{2^r}\sum\limits_{n=1}^{2^r}e^{\frac{2\pi i(m^2+n^2+mn)}{2^r}}}{2^r}$, an easy calculation shows $g(1)=-1,g(2)=1.$ When $r\geq3,$
\begin{align*}
  g(r) &=\frac{\sum\limits_{m=1}^{2^r}\sum\limits_{n=1}^{2^r}e^{\frac{2\pi i(m^2+n^2+mn)}{2^r}}}{2^r}\\
   & =\frac{\sum\limits_{m=1,m \,odd.}^{2^r}\sum\limits_{n=1}^{2^r}e^{\frac{2\pi i(m^2+n^2+mn)}{2^r}}}{2^r}
+\frac{\sum\limits_{m=1,m \, even.}^{2^r}\sum\limits_{n=1}^{2^r}e^{\frac{2\pi i(m^2+n^2+mn)}{2^r}}}{2^r}\\
 & =\frac{\sum\limits_{m=1,m \,odd.}^{2^r}e^{\frac{2\pi im^2}{2^r}}\sum\limits_{n=1.}^{2^r}e^{\frac{2\pi i(n^2+mn)}{2^r}}}{2^r}
+\frac{\sum\limits_{s=1,s=\frac{m}{2}.}^{2^{r-1}}\sum\limits_{n=1}^{2^r}e^{\frac{2\pi i(4s^2+n^2+2sn)}{2^r}}}{2^r}\\
& =\frac{\sum\limits_{m=1,m \,odd.}^{2^r}e^{\frac{2\pi im^2}{2^r}}\sum\limits_{t=1,t\,even.}^{2^r}2e^{\frac{2\pi it}{2^r}}}{2^r}
+\frac{\sum\limits_{n=1.}^{2^r}\sum\limits_{s=1}^{2^{r-1}}e^{\frac{2\pi i(4s^2+n^2+2sn)}{2^r}}}{2^r}\\
&=\frac{\sum\limits_{n=1}^{2^r}\sum\limits_{s=1}^{2^{r-1}}e^{\frac{2\pi i(4s^2+n^2+2sn)}{2^r}}}{2^r}\\
\end{align*}

For $\sum\limits_{n=1.}^{2^r}e^{\frac{2\pi i(n^2+mn)}{2^r}}
=\sum\limits_{t=1,t\,even.}^{2^r}2e^{\frac{2\pi it}{2^r}}$, we use the fact that when n runs a complete residue system modulo $2^r$, then
$t=n^2+mn$ runs a complete even residue system modulo $2^r$ twice if m is odd. First, when m is odd, $t=n^2+mn\equiv n^2+n=n(n+1)\equiv0$mod2 which shows that t must be even residue. Suppose $1\leq a,b\leq 2^r$, and $a^2+ma\equiv b^2+mb$ mod $2^r$, then $2^r| (a-b)(a+b+m)$. As $a-b$ and $a+b+m$ have different parities, then $2^r| (a-b)$ or  $2^r| (a+b+m)$ and only one can be true which means that for every even residue d mod$2^r$, we have two solutions of a which satisfies $a^2+ma\equiv d$ mod$2^r$, thus claim verified! Next 
\begin{align*}
g(r)&=\frac{\sum\limits_{n=1}^{2^{r-1}}\sum\limits_{s=1}^{2^{r-1}}e^{\frac{2\pi i(4s^2+n^2+2sn)}{2^r}}}{2^r}
+\frac{\sum\limits_{m=1,m=n-2^{r-1}.}^{2^{r-1}}\sum\limits_{s=1}^{2^{r-1}}e^{\frac{2\pi i(4s^2+(m+2^{r-1})^2+2s(m+2^{r-1}))}{2^r}}}{2^r}\\
&=\frac{\sum\limits_{n=1}^{2^{r-1}}\sum\limits_{s=1}^{2^{r-1}}e^{\frac{2\pi i(4s^2+n^2+2sn)}{2^r}}}{2^r}
+\frac{\sum\limits_{m=1.}^{2^{r-1}}\sum\limits_{s=1}^{2^{r-1}}e^{\frac{2\pi i(4s^2+m^2+2sm)}{2^r}}}{2^r}\\
&=\frac{\sum\limits_{n=1}^{2^{r-1}}\sum\limits_{s=1}^{2^{r-1}}e^{\frac{2\pi i(4s^2+n^2+2sn)}{2^r}}}{2^{r-1}}\\
&=\frac{\sum\limits_{n=1,n\,odd.}^{2^{r-1}}\sum\limits_{s=1}^{2^{r-1}}e^{\frac{2\pi i(4s^2+n^2+2sn)}{2^r}}}{2^{r-1}}
+\frac{\sum\limits_{n=1,n\,even.}^{2^{r-1}}\sum\limits_{s=1}^{2^{r-1}}e^{\frac{2\pi i(4s^2+n^2+2sn)}{2^r}}}{2^{r-1}}\\
\end{align*}
\begin{align*}
g(r)
&=\frac{\sum\limits_{n=1,n\,odd.}^{2^{r-1}}e^{\frac{2\pi in^2}{2^r}}\sum\limits_{s=1}^{2^{r-1}}e^{\frac{2\pi i(2s^2+sn)}{2^{r-1}}}}{2^{r-1}}
+\frac{\sum\limits_{t=1,t=\frac{n}{2}.}^{2^{r-2}}\sum\limits_{s=1}^{2^{r-1}}e^{\frac{2\pi i(4s^2+4t^2+2s2t)}{2^r}}}{2^{r-1}}\\
&=\frac{\sum\limits_{n=1,n\,odd.}^{2^{r-1}}e^{\frac{2\pi in^2}{2^r}}\sum\limits_{l=1,l=2s^2+sn.}^{2^{r-1}}e^{\frac{2\pi il}{2^{r-1}}}}{2^{r-1}}
+\frac{\sum\limits_{s=1.}^{2^{r-1}}\sum\limits_{t=1}^{2^{r-2}}e^{\frac{2\pi i(s^2+t^2+st)}{2^{r-2}}}}{2^{r-1}}\\
&=\frac{\sum\limits_{s=1.}^{2^{r-1}}\sum\limits_{t=1}^{2^{r-2}}e^{\frac{2\pi i(s^2+t^2+st)}{2^{r-2}}}}{2^{r-1}}\\
&=\frac{\sum\limits_{s=1.}^{2^{r-2}}\sum\limits_{t=1}^{2^{r-2}}e^{\frac{2\pi i(s^2+t^2+st)}{2^{r-2}}}}{2^{r-1}}
+\frac{\sum\limits_{m=1,m=s-2^{r-2}.}^{2^{r-2}}\sum\limits_{t=1}^{2^{r-2}}e^{\frac{2\pi i((m+2^{r-2})^2+t^2+(m+2^{r-2})t)}{2^{r-2}}}}{2^{r-1}}\\
\end{align*}

For $\sum\limits_{s=1}^{2^{r-1}}e^{\frac{2\pi i(2s^2+sn)}{2^{r-1}}}
=\sum\limits_{l=1,l=2s^2+sn.}^{2^{r-1}}e^{\frac{2\pi il}{2^{r-1}}}$, we use the fact that when s runs a complete residue system modulo $2^{r-1}$, then 
$l=2s^2+sn$ runs a complete residue system modulo $2^{r-1}$ if n is odd. Suppose $1\leq s,t\leq 2^{r-1}$, and $2s^2+sn\equiv2t^2+tn$ mod $2^{r-1}$, then $2^{r-1}| (s-t)(2(s+t)+n)$. As n is odd, then $2(s+t)+n$ is odd, hence  $2^{r-1}| (s-t)$. But $1\leq s,t\leq 2^{r-1}$, then s=t, thus claim verified!
\begin{align*}
 g(r)&=\frac{\sum\limits_{s=1.}^{2^{r-2}}\sum\limits_{t=1}^{2^{r-2}}e^{\frac{2\pi i(s^2+t^2+st)}{2^{r-2}}}}{2^{r-1}}
+\frac{\sum\limits_{m=1.}^{2^{r-2}}\sum\limits_{t=1}^{2^{r-2}}e^{\frac{2\pi i(m^2+t^2+mt)}{2^{r-2}}}}{2^{r-1}}\\
&=\frac{\sum\limits_{s=1.}^{2^{r-2}}\sum\limits_{t=1}^{2^{r-2}}e^{\frac{2\pi i(s^2+t^2+st)}{2^{r-2}}}}{2^{r-2}}\\
     &=g(r-2)\\
   \end{align*}

By induction, 
\[g(r)=\begin{cases}-1\quad\text{r odd}\\1\quad \text{r even}\end{cases}\].

Then \[e^{\frac{\pi ic}{4}}=g(r)=\begin{cases}-1\quad\text{r odd}\\1\quad \text{r even}\end{cases}\]

Therefore,
\[c\equiv\begin{cases}4\quad\text{r odd}\\0\quad \text{r even}\end{cases}\;mod\;8.\]

\end{itemize}
\end{proof}

\begin{cor}

\begin{enumerate}
    \item The central charge of an abelian anyon model in the $A,B,C,D$ classes is always an integer with the opposite parity of the order $|A|$ of $A$.
    \item The prime abelian anyon models come in conjugate pairs indexed by the central charges in Table \ref{cctable} that are complementary modulo $8$.
    \item Any $8$ copies of an abelian anyon model is a Drinfeld center.
\end{enumerate}

\end{cor}

The proofs are all straightforward inspections.

\subsection{Symmetries of all prime abelian anyon models}

The topological symmetry group of an anyon model is the decategorification of the categorical group of braided tensor auto-equivalences.  For abelian anyon models, they are the same as the isometry group of the metric group $(A,q)$: group isomorphisms $\gamma$ of $A$ to itself that preserve the quadratic form in the sense that $q(\gamma(x))\equiv q(x)$ mod 1.
In this subsection, we calculate the topological symmetry group $Aut((A,q))$ of the eight families of prime abelain anyon models.

First we need two lemmas.

\begin{lem}\label{p quadratic}
Let $p$ be an odd prime and $a\in \mathbb{Z}$ with $(a, p)=1.$ Then the equation
\[x^2\equiv a(mod \,p^k)\]
 \begin{itemize}
\item
either has no solution if $(\frac{a}{p})=-1$; \\
\item
or has two solutions $x_1$ and $-x_1$ if $(\frac{a}{p})=1$.
\end{itemize}
\end{lem}

\begin{lem}\label{2 quadratic}
Let $a$ be an odd integer. Then the following holds:
\begin{enumerate}
\item The equation
\[x^2\equiv a(mod \,2)\]
has the unique solution $x\equiv 1$ (mod \,2)
\item The equation 
either
 \begin{itemize}
\item
has no solution if $a \equiv 3(mod \,4)$; or\\
\item
has two solutions $x \equiv 1, 3(mod \,4)$ if $a \equiv 1(mod \,4)$.
\end{itemize}
\item When $k\geq 3$, the equation
\[x^2\equiv a(mod \,2^k)\]
either
 \begin{itemize}
\item
has no solution if $a \not \equiv 1(mod \,8)$; or\\
\item
has four solutions $x_1, -x_1,x_1+2^{k-1}, -(x_1+2^{k-1})$ if $a \equiv 1(mod \,8)$.
\end{itemize}

\end{enumerate}

\end{lem}

\begin{thm}\label{Aut}
The topological symmetry group of the eight families of prime abelian anyon models are given below where for the $F$ family, only the order of the symmetry group is found when $r\geq 4$.

\begin{enumerate}
    \item For the odd prime $A, B$ families: $$Aut(A_{p^r})=Aut(B_{p^r})\cong \mathbb{Z}_2$$
for $p$ odd prime and $r\geq 1$.

\item For the prime=$2$ families of $A,B,C,D$,
$$Aut(A_{2})=Aut(B_{2})\cong 1,$$
$$Aut(A_{2^r})=Aut(B_{2^r})=Aut(C_{2^r})=Aut(D_{2^r})\cong \mathbb{Z}_2$$ for  $r\geq 2$.

\item For the family $E$,

\[  Aut(E_{2^r})= \begin{cases}\mathbb{Z}_{2}, \quad r=1;\\
 \mathbb{Z}_{2}\times \mathbb{Z}_{2}, \quad r=2;\\
 (\mathbb{Z}_{2}\times\mathbb{Z}_{2^{r-2}})\rtimes \mathbb{Z}_{2}\quad r\geq3.\\
\end{cases} 
   \]

\item For the $F$ family,
$ Aut(F_{2})\cong D_3, Aut(F_{2^2})\cong D_6, Aut(F_{2^3})\cong D_6\rtimes \mathbb{Z}_2$.

$\left|Aut(F_{2^r})\right|=3\times2^r$ for $r\geq 4$.
\end{enumerate}

\begin{proof}
  $\forall  f \in Aut(A_{p^r})$, where $p$ is odd prime and $r\geq 1$.
  As $\mathbb{Z}_{p^r}$ is cyclic , so $f$ is determined by its image $f(1)$. Suppose $f(1)=j,$ then $q(f(1))\equiv q(1) (mod \,1)$ is equivalent to $\frac{mj^2}{p^r}\equiv \frac{m}{p^r} (mod \,1)$. As $(m, p)=1$, so it reduced to $j^2\equiv 1(mod \, p^r)$. According to Lem\ref{p quadratic}, the equation has two solutions for $j$ as $(\frac{1}{p})=1$ always holds. The two solutions for $j$ is $\pm 1$, so 
  $Aut(A_{p^r})=\mathbb{Z}_2$  . Parallel to the proof above, we easily get 
  $Aut(B_{p^r})=\mathbb{Z}_2$ .

 $\forall  f \in Aut(A_{2^r})$, as $\mathbb{Z}_{2^r}$ is cyclic , so $f$ is determined by its image $f(1)$. Suppose $f(1)=j,$ then $q(f(1))\equiv q(1) (mod \,1)$ is equivalent to $\frac{j^2}{2^{r+1}}\equiv \frac{1}{2^{r+1}} (mod \,1)$. It reduced to $j^2\equiv 1(mod \, 2^{r+1})$. According to Lem\ref{2 quadratic}, when $r=1, j$ has one solution $1$ which means  $Aut(A_{2})=1$; when $r\geq 2 ,$ j has four solutions $1, -1, 1+2^r, -(1+2^r)$. But $1$ and $1+2^r$ are the same in $Z_{2^r}$ which means $Aut(A_{2^r})=\mathbb{Z}_2$ when $r\geq 2.$ Parallel to the proof of  $Aut(A_{2^r})$, we can easily get $Aut(B_{2})=Aut(A_{2})=1$ and $Aut(B_{2^r})=Aut(C_{2^r})=Aut(D_{2^r})=Aut(A_{2^r})=\mathbb{Z}_2$ for  $r\geq 2$ .
 
    $\forall  f \in Aut(E_{2^r})$, as $\mathbb{Z}_{2^r}\times \mathbb{Z}_{2^r}$ has  two generators , suppose $f(e_1)=a e_1+b e_2, f(e_2)=c e_1+d e_2$ where $e_1=(1,0)$ and $e_2=(0,1)$ and $0\leq a, b, c, d \leq 2^r-1$ , then some necessary conditions are $q(f(e_1))\equiv q(e_1) (mod \,1)$  and $q(f(e_2))\equiv q(e_2) (mod \,1)$ and the determinant of transformation is invertible. All these conditions are equivalent to $\frac{ab}{2^{r}}\equiv 0 (mod \,1)$ and $\frac{cd}{2^{r}}\equiv 0 (mod \,1)$ and $ad-bc$ is odd. It reduced to $ab\equiv 0(mod \, 2^{r})$ and $cd\equiv 0(mod \, 2^{r})$ and $ad-bc$ is odd. If $a$ and $b$ are all even, then $ad-bc$ is even which is a contradiction, if $a$ and $b$ are all odd, then $ab\equiv 0(mod \, 2^{r})$ can not happen, so $a$ and $b$ have different parities. As $ab\equiv 0(mod \, 2^{r})$, so one of them must be 0. Suppose $a=0$, then parallel to the analysis of  $a$ and $b$ and $ad-bc$ is odd, we have $d=0$ and $b$ and $c$ are all odd. As $q(f(e_{(m,n)}))=q(e_{(m,n)})$ where 
  $e_{(m,n)}=m e_1+n e_2$, so $q(c n e_1+b m e_2)=q(m e_1+n e_2)$ which reduced to $\frac{b c m n}{2^r}\equiv \frac{m n}{2^r} ( mod\, 1)$ for any $0\leq m, n \leq 2^r-1$. So $b c\equiv 1(mod \,2^r)$ which means $b$ and $c$ are inverse to each other in the sense of mod $2^r$.  When $a=d=0,$ the transformation matrix $\begin{pmatrix}0&b\\
  b^{-1}(mod \, 2^r)&0  \end{pmatrix}$ corresponds to an element of order 2 in $Aut(E_{2^r})$. We define $B=\{B_b: B_b=\begin{pmatrix}0&b\\
  b^{-1}(mod \, 2^r)&0  \end{pmatrix}\text{ where $b$ runs a complete residue of } {\mathbb{Z}_{2^r}}^*\}$ . When $b=c=0,$ the transformation matrix $\begin{pmatrix}a&0\\
  0 & a^{-1}(mod \, 2^r)\end{pmatrix}$ composite the  group
   \[ 
 \begin{cases}1, \quad r=1;\\
 \mathbb{Z}_{2}, \quad r=2;\\
 \mathbb{Z}_{2}\times\mathbb{Z}_{2^{r-2}}\quad r\geq3.\\
\end{cases} 
 \]We define $A=\{A_a: A_a=\begin{pmatrix}a&0
  \\
 0 &a^{-1}(mod \, 2^r)  \end{pmatrix}\text{ where $a$ runs a complete residue of } {\mathbb{Z}_{2^r}}^*\}$ . Thus $Aut(E_{2^r})=\{A_a, B_b: \text{ where $a$ and $b$ run a complete residue of } {\mathbb{Z}_{2^r}}^*\}$. 
 
 As we have $A_{a_1} A_{a_2}=A_{a_1 a_2}, B_{b_1} B_{b_2}=A_{b_1 b_2^{-1}}, A_a B_b=B_{a b}$, so 
 $Aut(E_{2^r})=<A_a, B_1:  \text{ where $a$  runs a complete residue of } {\mathbb{Z}_{2^r}}^*>$. As $A$ is a group of order $2^{r-1}$ and has index 2 in $Aut(E_{2^r})$
  , so $A\unlhd Aut(E_{2^r})$. Besides, every element of $Aut(E_{2^r})$, if it is some $A_a$, it equals $A_a\times B_1^{2}$; if it is some $B_b$, it equals $A_b\times B_1$
 which means $Aut(E_{2^r})=A<B_1>$, as $A\cap <B_1>=1, $so 
 \[  Aut(E_{2^r})=A \rtimes\mathbb{Z}_{2}= \begin{cases}\mathbb{Z}_{2}, \quad r=1;\\
 \mathbb{Z}_{2}\times \mathbb{Z}_{2}, \quad r=2;\\
 (\mathbb{Z}_{2}\times\mathbb{Z}_{2^{r-2}})\rtimes\mathbb{Z}_{2}\quad r\geq3.\\
\end{cases} 
   \]

  $\forall  f \in Aut(F_{2^r})$, as $\mathbb{Z}_{2^r}\times \mathbb{Z}_{2^r}$ has  two generators , suppose $f(e_1)=a e_1+b e_2, f(e_2)=c e_1+d e_2$ where $e_1=(1,0)$ and $e_2=(0,1)$ and $0\leq a, b, c, d \leq 2^r-1$ .  As $q(f(e_{(m,n)}))=q(e_{(m,n)})$ where 
  $e_{(m,n)}=m e_1+n e_2$, so $q((a m+c n) e_1+(b m+d n) e_2)=q(m e_1+n e_2)$ which reduced to $\frac{m^2(a^2+b^2+a b)+n^2(c^2+d^2+cd)+(2 a c +2 b d +a d +b c )m n}{2^r}\equiv \frac{m^2+n^2+m n}{2^r} ( mod\, 1)$ for any $0\leq m, n \leq 2^r-1$. So 
  \[
  \begin{cases}
  a^2+b^2+a b\equiv 1(mod \,2^r)\\
  c^2+d^2+c d \equiv 1(mod \,2^r)\\
  2 a c +2 b d +a d +b c \equiv 1(mod \,2^r)\\
  a d-b c \text{ is odd}
  \end{cases}
  \]
  
  When $r=1$,  \[
  \begin{cases}
  a^2+b^2+a b\equiv 1(mod \,2)\\
  c^2+d^2+c d \equiv 1(mod \,2)\\
  2 a c +2 b d +a d +b c \equiv 1(mod \,2)\\
  a d-b c \text{ is odd}
  \end{cases}
  \]
it reduced to
\[
  \begin{cases}
  a^2+b^2+a b\equiv 1(mod \,2)\\
  c^2+d^2+c d \equiv 1(mod \,2)\\
  a d +b c \equiv 1(mod \,2)\\
   \end{cases}
  \]

so $a$ and $b$ can not be both even, so they can be both odd or have different 
parities, a complete discussion shows that the solution of 
$\begin{pmatrix}a&b\\
  c&d\end{pmatrix}$  has 6 elements 
  
  \[A=\begin{pmatrix}1&1\\
  1&0\end{pmatrix}, B=\begin{pmatrix}0&1\\
  1&1\end{pmatrix}, C=\begin{pmatrix}1&0\\
  0&1\end{pmatrix},  D=\begin{pmatrix}1&1\\
  0&1\end{pmatrix}, E=\begin{pmatrix}1&0\\
  1&1\end{pmatrix}, F=\begin{pmatrix}0&1\\
  1&0\end{pmatrix}.
  \]
Calculation shows $A^3=C=1, F^2=1, F A F^{-1}=A^{-1}, A^2=B, A F =D, A^2 	F=E, A^3 	F=F$. So $Aut(F_{2})=D_{3}$, the dihedral group of order 6.

When $r=2$,  \[
  \begin{cases}
  a^2+b^2+a b\equiv 1(mod \,4)\\
  c^2+d^2+c d \equiv 1(mod \,4)\\
  2 a c +2 b d +a d +b c \equiv 1(mod \,4)\\
  a d-b c \text{ is odd}
  \end{cases}
  \]
  
so $a$ and $b$ can not be both even, so they can be both odd or have different 
parities. When they are both odd, they have different residue module 4; when they have different parities, one of them is 0 and another is odd. A complete discussion shows that the solution of 
$\begin{pmatrix}a&b\\
  c&d\end{pmatrix}$  has 12 elements 
  
  \[A=\begin{pmatrix}0&1\\
  -1&1\end{pmatrix}, B=\begin{pmatrix}-1&1\\
  -1&0\end{pmatrix}, C=\begin{pmatrix}-1&0\\
  0&-1\end{pmatrix}, D=\begin{pmatrix}0&-1\\
  1&-1\end{pmatrix}, E=\begin{pmatrix}1&-1\\
  1&0\end{pmatrix}, F=\begin{pmatrix}1&0\\
  0&1\end{pmatrix},\]
  \[ G=\begin{pmatrix}1&0\\
  1&-1\end{pmatrix}, H=\begin{pmatrix}1&-1\\
  0&-1\end{pmatrix}, I=\begin{pmatrix}0&-1\\
  -1&0\end{pmatrix}, J=\begin{pmatrix}-1&0\\
  -1&1\end{pmatrix}, K=\begin{pmatrix}-1&1\\
  0&1\end{pmatrix}, L=\begin{pmatrix}0&1\\
  1&0\end{pmatrix}.\]

Calculation shows $A^6=F=1, L^2=1, L A L^{-1}=A^{-1}, A^2=B, A^3=C, A^4=D, A^5=E,  A L =G, A^2 L=H,  A^3 L=I,  A^4 L=J,  A^5 L=K, A^6 L=L$. So $Aut(F_{2^2})=D_{6}$, the dihedral group of order 12.

 When $r\geq 3$, we can verify that the matrix $A, L$ above also satisfy conditions associate with (mod $2^r$) which means $D_{6}\leq Aut(F_{2^r})$, we will show $Aut(F_{2^3})=D_6\rtimes \mathbb{Z}_2$. 

When $r=3$,  \[
  \begin{cases}
  a^2+b^2+a b\equiv 1(mod \,8)\\
  c^2+d^2+c d \equiv 1(mod \,8)\\
  2 a c +2 b d +a d +b c \equiv 1(mod \,8)\\
  a d-b c \text{ is odd}
  \end{cases}
  \]
We rewrite the first equation to be $(a+b)^2\equiv a b+1 (mod \,8)$. According to Lem\ref{2 quadratic},  when $a b$ is even, $1+a b$ is odd, the existence of $a, b$ requires $1+a b\equiv 1(mod\, 8)$ which means $a b\equiv 0(mod\,8)$, as $a+b$ is odd, so one of them must be 0 and another one is odd, we can check that all such pairs satisfy $a^2+b^2+a b\equiv 1(mod \,8)$. When $a b$ is odd, $a$ and $b$ are all odd, $a^2\equiv b^2\equiv 1(mod\,8)$, so we have $a b\equiv -1(mod\, 8)$ , we can check that $a=-b$ and all such pairs satisfy $a^2+b^2+a b\equiv 1(mod \,8)$.A complete discussion shows that the solution of 
$\begin{pmatrix}a&b\\
  c&d\end{pmatrix}$  has 24 elements 
  
  \[A=\begin{pmatrix}0&1\\
  -1&1\end{pmatrix}, B=\begin{pmatrix}-1&1\\
  -1&0\end{pmatrix}, C=\begin{pmatrix}-1&0\\
  0&-1\end{pmatrix}, D=\begin{pmatrix}0&-1\\
  1&-1\end{pmatrix}, E=\begin{pmatrix}1&-1\\
  1&0\end{pmatrix}, F=\begin{pmatrix}1&0\\
  0&1\end{pmatrix},\]
  \[ G=\begin{pmatrix}1&0\\
  1&-1\end{pmatrix}, H=\begin{pmatrix}1&-1\\
  0&-1\end{pmatrix}, I=\begin{pmatrix}0&-1\\
  -1&0\end{pmatrix}, J=\begin{pmatrix}-1&0\\
  -1&1\end{pmatrix}, K=\begin{pmatrix}-1&1\\
  0&1\end{pmatrix}, L=\begin{pmatrix}0&1\\
  1&0\end{pmatrix}.\]
  \[M=\begin{pmatrix}3&-3\\
  0&-3\end{pmatrix}, N=\begin{pmatrix}0&-3\\
  -3&0\end{pmatrix}, O=\begin{pmatrix}-3&0\\
  -3&3\end{pmatrix}, P=\begin{pmatrix}-3&3\\
  0&3\end{pmatrix}, Q=\begin{pmatrix}0&3\\
  3&0\end{pmatrix}, R=\begin{pmatrix}3&0\\
  3&-3\end{pmatrix},\]
  \[ S=\begin{pmatrix}3&0\\
  0&3\end{pmatrix}, T=\begin{pmatrix}0&3\\
  -3&3\end{pmatrix}, U=\begin{pmatrix}-3&3\\
  -3&0\end{pmatrix}, V=\begin{pmatrix}-3&0\\
  0&-3\end{pmatrix}, W=\begin{pmatrix}0&-3\\
  3&-3\end{pmatrix}, X=\begin{pmatrix}3&-3\\
  3&0\end{pmatrix}.\]
  
Calculation shows $A^6=F=1, L^2=1, L A L^{-1}=A^{-1}, A^2=B, A^3=C, A^4=D, A^5=E,  A L =G, A^2 L=H,  A^3 L=I,  A^4 L=J,  A^5 L=K, A^6 L=L$. So $<A, L>=D_{6}\leq Aut(F_{2^3})$. As $D_6$ has an index 2, so 
$D_6\unlhd Aut(F_{2^3})$. We show $Aut(F_{2^3})=D_6\rtimes \mathbb{Z}_2$. Take $R$ as a generator of $\mathbb{Z}_2$, we have 
\[A R=M, A^2 R=N,  A^3 R=O, A^4 R=P, A^5 R=Q, A^6 R=R,\]
\[A L R=S, A^2 L R=T, A^3 L R=U, A^4 L R=V, A^5 L R=W, A^6 L R=X.\]
So  $Aut(F_{2^3})=<A, L><R>=D_{6}\mathbb{Z}_2$, as $D_6$ is normal and 
$D_{6}\cap \mathbb{Z}_2=1$, so $Aut(F_{2^3})=D_6\rtimes \mathbb{Z}_2$. 

In general, we have  $\left|Aut(F_{2^r})\right|=3\times2^r$. Let us try to lift every solution of  $Aut(F_{2^r})$ to two solutions of $Aut(F_{2^{r+1}})$ and all lifted solutions are different, thus we would have   $\left|Aut(F_{2^r})\right|=3\times2^r$
 by induction.
  
If $\begin{pmatrix}\bar{a}&\bar{b}\\
  \bar{c}&\bar{d}\end{pmatrix}$  is a solution of  $Aut(F_{2^{r+1}})$, then   $\begin{pmatrix}a&b\\
  c&d\end{pmatrix}
 $ is a solution of  $Aut(F_{2^{r}})$ where $\bar{a}\equiv a (mod \,2^r), \bar{b}\equiv b (mod \,2^r), \bar{c}\equiv c (mod \,2^r), \bar{d}\equiv d (mod \,2^r)$. It means every solution of $Aut(F_{2^{r+1}})$ can be obtained by lifting the solution of $Aut(F_{2^{r}})$. We want to explore how many ways the lifting can be.
 This is to say if we have a solution $\begin{pmatrix}a&b\\
  c&d\end{pmatrix}
 $ of  $Aut(F_{2^{r}})$, then what element should be lifted to get a solution of 
  $Aut(F_{2^{r+1}})$. Here, by saying lifting, we mean by adding $2^r$. If the element do not need lifted, we will keep the symbol  the same one in the solution of 
   $Aut(F_{2^{r+1}})$, otherwise we will use the lifted symbol such as $\bar{a}=a+2^r$.
 
 When we consider such lifting of solutions from   $Aut(F_{2^{r}})$, we need to split the conditions to be three one. The fourth one $a d -b c$ is odd does not need to consider because it is always true when we do such lifting.
 
 \begin{equation}\label{1}
  a^2+b^2+a b\overset{?}{\equiv}1 (mod \,2^{r+1})
  \end{equation}
 \begin{equation}\label{2}
  c^2+d^2+c d \overset{?}\equiv 1(mod \,2^{r+1})
  \end{equation}

 \begin{equation}\label{3}
   2 a c+2 b d+a d +b c \overset{?}\equiv 1(mod \,2^{r+1})
\end{equation}
 
 We know $\begin{pmatrix}a&b\\
  c&d\end{pmatrix}
 $ is the solution of   $Aut(F_{2^{r}})$, but it need not to be the solution of 
   $Aut(F_{2^{r+1}})$ . We need to consider the parity of $a, b, c, d$ when we do such lifting. For convenience and simplicity, we use $1$ to represent the element is odd and use  $0$ to represent the element is even. Then the type 
$ \begin{pmatrix}0&1\\
  1&0\end{pmatrix}$ means $a, d$ are even and $b, c$ are odd. There are six types of them which correspondents to $\left|GL(2, \mathbb{Z}_2)\right|=6$. 
 We call $ \begin{pmatrix}1&0\\
  0&1\end{pmatrix}$, $ \begin{pmatrix}1&0\\
  1&1\end{pmatrix}$, $ \begin{pmatrix}0&1\\
  1&0\end{pmatrix}$, $ \begin{pmatrix}0&1\\
  1&1\end{pmatrix}$, $ \begin{pmatrix}1&1\\
  1&0\end{pmatrix}$, $ \begin{pmatrix}1&1\\
  0&1\end{pmatrix}$  type $A, B, C, D, E, F$ respectively. In fact, we just need to consider the lifting of type $A, B, C$ because the lifting of type $D, E, F$
 can be easily obtained by the lifting of type $B$, just consider the symmetry.
 
 Next, we list how to lift according to the type and how many conditions they satisfy in $Aut(F_{2^{r+1}})$. 
 \begin{table}[!htbp]\label{tablec}
 \renewcommand\arraystretch{1.5}
  \centering 
 \caption{Solutions lifting from $Aut(F_{2^{r}})$ to $Aut(F_{2^{r+1}})$}  
 
\begin{tabular}{|c|c|c|c|c|}

\hline
{}&\eqref{1}, \eqref{2},\eqref{3}&\eqref{1}, \eqref{2}&\eqref{1}, \eqref{3}&
\eqref{2},\eqref{3}\\

\hline
A=$ \begin{pmatrix}1&0\\
  0&1\end{pmatrix}$&$\begin{pmatrix}a&b\\
  c&d\end{pmatrix}$, $\begin{pmatrix}\bar{a}&b\\
  c&\bar{d}\end{pmatrix}$&$\begin{pmatrix}\bar{a}&b\\
  c&d\end{pmatrix}$, $\begin{pmatrix}a&b\\
  c&\bar{d}\end{pmatrix}$&$\begin{pmatrix}a&b\\
  \bar{c}&d\end{pmatrix}$, $\begin{pmatrix}\bar{a}&b\\
  \bar{c}&\bar{d}\end{pmatrix}$&$\begin{pmatrix}\bar{a}&b\\
  c&d\end{pmatrix}$, $\begin{pmatrix}\bar{a}&\bar{b}\\
  \bar{c}&d\end{pmatrix}$\\
\hline
B=$ \begin{pmatrix}1&0\\
  1&1\end{pmatrix}$&$\begin{pmatrix}a&b\\
  c&d\end{pmatrix}$, $\begin{pmatrix}\bar{a}&b\\
  \bar{c}&\bar{d}\end{pmatrix}$&$\begin{pmatrix}\bar{a}&b\\
  c&d\end{pmatrix}$, $\begin{pmatrix}a&b\\
  \bar{c}&\bar{d}\end{pmatrix}$&$\begin{pmatrix}a&b\\
  \bar{c}&d\end{pmatrix}$, $\begin{pmatrix}\bar{a}&b\\
  c&\bar{d}\end{pmatrix}$&$\begin{pmatrix}\bar{a}&b\\
  c&d\end{pmatrix}$, $\begin{pmatrix}a&\bar{b}\\
  \bar{c}&d\end{pmatrix}$\\

\hline
C=$ \begin{pmatrix}0&1\\
  1&0\end{pmatrix}$&$\begin{pmatrix}a&b\\
  c&d\end{pmatrix}$, $\begin{pmatrix}a&\bar{b}\\
  \bar{c}&d\end{pmatrix}$&$\begin{pmatrix}a&\bar{b}\\
  c&d\end{pmatrix}$, $\begin{pmatrix}a&b\\
  \bar{c}&d\end{pmatrix}$&$\begin{pmatrix}a&b\\
  c&\bar{d}\end{pmatrix}$, $\begin{pmatrix}a&\bar{b}\\
  \bar{c}&\bar{d}\end{pmatrix}$&$\begin{pmatrix}a&\bar{b}\\
  c&d\end{pmatrix}$, $\begin{pmatrix}\bar{a}&\bar{b}\\
  c&\bar{d}\end{pmatrix}$\\
\hline
D=$ \begin{pmatrix}0&1\\
  1&1\end{pmatrix}$&$\begin{pmatrix}a&b\\
  c&d\end{pmatrix}$, $\begin{pmatrix}a&\bar{b}\\
  \bar{c}&\bar{d}\end{pmatrix}$&$\begin{pmatrix}a&\bar{b}\\
  c&d\end{pmatrix}$, $\begin{pmatrix}a&b\\
  \bar{c}&\bar{d}\end{pmatrix}$& $\begin{pmatrix}a&b\\
  c&\bar{d}\end{pmatrix}$, $\begin{pmatrix}a&\bar{b}\\
  \bar{c}&d\end{pmatrix}$& $\begin{pmatrix}a&\bar{b}\\
  c&d\end{pmatrix}$, $\begin{pmatrix}\bar{a}&b\\
c&\bar{d}\end{pmatrix}$\\

\hline
E=$ \begin{pmatrix}1&1\\
  1&0\end{pmatrix}$&$\begin{pmatrix}a&b\\
  c&d\end{pmatrix}$, $\begin{pmatrix}\bar{a}&\bar{b}\\
  \bar{c}&d\end{pmatrix}$&$\begin{pmatrix}a&b\\
  \bar{c}&d\end{pmatrix}$, $\begin{pmatrix}\bar{a}&\bar{b}\\
  c&d\end{pmatrix}$&$\begin{pmatrix}\bar{a}&b\\
  c&d\end{pmatrix}$, $\begin{pmatrix}a&\bar{b}\\
  \bar{c}&d\end{pmatrix}$&$\begin{pmatrix}a&b\\
  \bar{c}&d\end{pmatrix}$, $\begin{pmatrix}\bar{a}&b\\
  c&\bar{d}\end{pmatrix}$\\

\hline
F=$ \begin{pmatrix}1&1\\
  0&1\end{pmatrix}$&$\begin{pmatrix}a&b\\
  c&d\end{pmatrix}$, $\begin{pmatrix}\bar{a}&\bar{b}\\
 c&\bar{d}\end{pmatrix}$&$\begin{pmatrix}a&b\\
  c&\bar{d}\end{pmatrix}$,$\begin{pmatrix}\bar{a}&\bar{b}\\
  c&d\end{pmatrix}$&$\begin{pmatrix}a&\bar{b}\\
  c&d\end{pmatrix}$, $\begin{pmatrix}\bar{a}&b\\
  c&\bar{d}\end{pmatrix}$&$\begin{pmatrix}a&b\\
  c&\bar{d}\end{pmatrix}$, $\begin{pmatrix}a&\bar{b}\\
  \bar{c}&d\end{pmatrix}$\\

\hline
{}&\eqref{1}&\eqref{2}&\eqref{3}&$\times$\\

\hline
A=$ \begin{pmatrix}1&0\\
  0&1\end{pmatrix}$&$\begin{pmatrix}\bar{a}&b\\
  \bar{c}&d\end{pmatrix}$, $\begin{pmatrix}a&b\\
  \bar{c}&\bar{d}\end{pmatrix}$&$\begin{pmatrix}\bar{a}&b\\
  \bar{c}&d\end{pmatrix}$, $\begin{pmatrix}\bar{a}&\bar{b}\\
  c&d\end{pmatrix}$&$\begin{pmatrix}a&\bar{b}\\
  \bar{c}&d\end{pmatrix}$, $\begin{pmatrix}\bar{a}&\bar{b}\\
  \bar{c}&\bar{d}\end{pmatrix}$&$\begin{pmatrix}\bar{a}&\bar{b}\\
  \bar{c}&d\end{pmatrix}$, $\begin{pmatrix}a&\bar{b}\\
  \bar{c}&\bar{d}\end{pmatrix}$\\
\hline
B=$ \begin{pmatrix}1&0\\
  1&1\end{pmatrix}$&$\begin{pmatrix}a&b\\
  c&\bar{d}\end{pmatrix}$, $\begin{pmatrix}\bar{a}&b\\
  \bar{c}&d\end{pmatrix}$&$\begin{pmatrix}a&\bar{b}\\
  c&d\end{pmatrix}$, $\begin{pmatrix}\bar{a}&b\\
  \bar{c}&d\end{pmatrix}$&$\begin{pmatrix}\bar{a}&\bar{b}\\
  \bar{c}&d\end{pmatrix}$, $\begin{pmatrix}a&\bar{b}\\
  c&\bar{d}\end{pmatrix}$&$\begin{pmatrix}a&\bar{b}\\
  \bar{c}&d\end{pmatrix}$, $\begin{pmatrix}\bar{a}&\bar{b}\\
  \bar{c}&d\end{pmatrix}$\\

\hline
C=$ \begin{pmatrix}0&1\\
  1&0\end{pmatrix}$&$\begin{pmatrix}a&\bar{b}\\
  c&\bar{d}\end{pmatrix}$, $\begin{pmatrix}a&b\\
  \bar{c}&\bar{d}\end{pmatrix}$&$\begin{pmatrix}a&\bar{b}\\
  c&\bar{d}\end{pmatrix}$, $\begin{pmatrix}\bar{a}&\bar{b}\\
  c&d\end{pmatrix}$& $\begin{pmatrix}\bar{a}&b\\
  c&\bar{d}\end{pmatrix}$, $\begin{pmatrix}\bar{a}&\bar{b}\\
  \bar{c}&\bar{d}\end{pmatrix}$&$\begin{pmatrix}\bar{a}&\bar{b}\\
  c&\bar{d}\end{pmatrix}$, $\begin{pmatrix}\bar{a}&b\\
  \bar{c}&\bar{d}\end{pmatrix}$\\
\hline
D=$ \begin{pmatrix}0&1\\
  1&1\end{pmatrix}$&$\begin{pmatrix}a&b\\
  \bar{c}&d\end{pmatrix}$,  $\begin{pmatrix}a&\bar{b}\\
  c&\bar{d}\end{pmatrix}$&$\begin{pmatrix}\bar{a}&b\\
  c&d\end{pmatrix}$,  $\begin{pmatrix}a&\bar{b}\\
  c&\bar{d}\end{pmatrix}$&$\begin{pmatrix}\bar{a}&\bar{b}\\
  c&\bar{d}\end{pmatrix}$, $\begin{pmatrix}\bar{a}&b\\
  \bar{c}&d\end{pmatrix}$&$\begin{pmatrix}\bar{a}&b\\
  c&\bar{d}\end{pmatrix}$, $\begin{pmatrix}\bar{a}&\bar{b}\\
  c&\bar{d}\end{pmatrix}$\\

\hline
E=$ \begin{pmatrix}1&1\\
  1&0\end{pmatrix}$&$\begin{pmatrix}a&\bar{b}\\
  c&d\end{pmatrix}$, $\begin{pmatrix}\bar{a}&b\\
  \bar{c}&d\end{pmatrix}$&$\begin{pmatrix}a&b\\
  c&\bar{d}\end{pmatrix}$, $\begin{pmatrix}\bar{a}&b\\
  \bar{c}&d\end{pmatrix}$&$\begin{pmatrix}\bar{a}&b\\
  \bar{c}&\bar{d}\end{pmatrix}$, $\begin{pmatrix}a&\bar{b}\\
  c&\bar{d}\end{pmatrix}$&$\begin{pmatrix}\bar{a}&b\\
  c&\bar{d}\end{pmatrix}$, $\begin{pmatrix}\bar{a}&b\\
  \bar{c}&\bar{d}\end{pmatrix}$\\

\hline
F=$ \begin{pmatrix}1&1\\
  0&1\end{pmatrix}$&$\begin{pmatrix}\bar{a}&b\\
  c&d\end{pmatrix}$,  $\begin{pmatrix}a&\bar{b}\\
  c&\bar{d}\end{pmatrix}$&$\begin{pmatrix}a&b\\
  \bar{c}&d\end{pmatrix}$,  $\begin{pmatrix}a&\bar{b}\\
  c&\bar{d}\end{pmatrix}$&$\begin{pmatrix}a&\bar{b}\\
  \bar{c}&\bar{d}\end{pmatrix}$, $\begin{pmatrix}\bar{a}&b\\
  \bar{c}&d\end{pmatrix}$&$\begin{pmatrix}a&\bar{b}\\
  \bar{c}&d\end{pmatrix}$, $\begin{pmatrix}a&\bar{b}\\
  \bar{c}&\bar{d}\end{pmatrix}$\\
\hline

\end{tabular}

\end{table}
 
\begin{table}[!htbp]\label{tablec1}
 \renewcommand\arraystretch{1.5}
  \centering 
 \caption{Solutions lifting from $Aut(F_{2^2})$ to $Aut(F_{2^3})$}  
 
\begin{tabular}{|c|c|c|c|c|}

\hline
{}&\eqref{1}, \eqref{2},\eqref{3}&\eqref{1}, \eqref{2}&\eqref{1}, \eqref{3}&
\eqref{2},\eqref{3}\\

\hline
A=$ \begin{pmatrix}1&0\\
  0&1\end{pmatrix}$&{}&{}&{}&{}\\
\hline
B=$ \begin{pmatrix}1&0\\
  1&1\end{pmatrix}$&{}&{}&{}&{}\\

\hline
C=$ \begin{pmatrix}0&1\\
  1&0\end{pmatrix}$&{}&{}&{}&{}\\
\hline
D=$ \begin{pmatrix}0&1\\
  1&1\end{pmatrix}$&{}&{}&{}&{}\\

\hline
E=$ \begin{pmatrix}1&1\\
  1&0\end{pmatrix}$&{}&{}&{}&{}\\

\hline
F=$ \begin{pmatrix}1&1\\
  0&1\end{pmatrix}$&{}&{}&{}&{}\\

\hline
{}&\eqref{1}&\eqref{2}&\eqref{3}&$\times$\\

\hline
A=$ \begin{pmatrix}1&0\\
  0&1\end{pmatrix}$&{}&{}&$C\rightarrow{S, C}; F\rightarrow{F, V}$&{}\\
\hline
B=$ \begin{pmatrix}1&0\\
  1&1\end{pmatrix}$&{}&{}&$G\rightarrow{O, G}; J\rightarrow{J, R}$&{}\\

\hline
C=$ \begin{pmatrix}0&1\\
  1&0\end{pmatrix}$&{}&{}&$I\rightarrow{Q, I}; L\rightarrow{L, N}$&{}\\
\hline
D=$ \begin{pmatrix}0&1\\
  1&1\end{pmatrix}$&{}&{}&$A\rightarrow{W, A}; D\rightarrow{D, T}$&{}\\

\hline
E=$ \begin{pmatrix}1&1\\
  1&0\end{pmatrix}$&{}&{}&$B\rightarrow{B, X}; E\rightarrow{U, E}$&{}\\
\hline
F=$ \begin{pmatrix}1&1\\
  0&1\end{pmatrix}$&{}&{}&$H\rightarrow{H, P}; K\rightarrow{M, K}$&{}\\\hline

\end{tabular}

\end{table}
 
From the chart, we know every solution of  $Aut(F_{2^r})$ can be lifted to two solutions of $Aut(F_{2^{r+1}})$, and for a fixed type, the liftings are all different; for different types, the liftings are also different because of the parities of elements. Thus all of the lifted solutions are different, so our claim verified.
 
 Let us take $r=2$ for example, we will lift every solution in $Aut(F_{2^2})$ to two solutions in $Aut(F_{2^3})$. 
 
 \end{proof}
\end{thm}
 
\subsection{Gauging symmetries of prime  abelian anyon models}
It would be very interesting to classify the anyon models obtained from gauging the symmetries of prime abelian anyon models \cite{CGPW,CSW}.  But as the three fermion example shows, it is in general a very difficult problem \cite{CGPW}.  In this subsection, we derive the total quantum dimensions of the resulting anyon models and then Property $F$ for those modular categories follows.

\begin{prop}\label{FPdim prime gauging}

The braided $G$-crossed extension of $\mathcal{B}=(A,q)$ is denoted as  $\mathcal{B}_G^{\times}$, where $G=Aut(\mathcal{B})$. Let $\mathcal{C_B}=
\mathcal{B}_G^{\times, G}$ be the G-equivariantization of $\mathcal{B}_G^{\times}$, the anyon model $\mathcal{C_B}$ is called the $G$-gauging of  $\mathcal{B}$.
Then $\mathcal{C_B}$ has Property F for prime anyon models $\mathcal{B}$ follows from the following:

$FPdim(\mathcal{Z(C_B)})=\left|G\right|^4(FPdim(\mathcal{B}))^2$, and specifically

$FPdim(\mathcal{Z}(\mathcal{C}_{A_{p^r}}))=FPdim(\mathcal{Z}(\mathcal{C}_{B_{p^r}}))=2^{4}p^{2 r}$;

$FPdim(\mathcal{Z}(\mathcal{C}_{A_{2}}))=FPdim(\mathcal{Z}(\mathcal{C}_{B_{2}}))=2^2$;

$FPdim(\mathcal{Z}(\mathcal{C}_{A_{2^r}}))=FPdim(\mathcal{Z}(\mathcal{C}_{B_{2^r}}))=FPdim(\mathcal{Z}(\mathcal{C}_{C_{2^r}}))=FPdim(\mathcal{Z}(\mathcal{C}_{D_{2^r}}))=2^{2 r+4}$ for $r\geq2$;

$FPdim(\mathcal{Z}(\mathcal{C}_{E_{2^r}}))=2^{8 r}$;

$FPdim(\mathcal{Z}(\mathcal{C}_{F_{2^r}}))=2^{8 r}3^4$ for $r\geq 1$.
\end{prop}

\begin{proof}

$FPdim(\mathcal{Z(C_B)})=\left|G\right|^2FPdim(\mathcal{Z}(\mathcal{B}_G^{\times}))=\left|G\right|^2\left|G\right|^2FPdim(\mathcal{Z}(\mathcal{B}))
=\left|G\right|^4(FPdim(\mathcal{B}))^2.$

Since $\mathcal{C_B}$ is braided, there exists a canonical braided tensor functor $F: \mathcal{C_B}\rightarrow \mathcal{Z(C_B)}$. Thus it suffices to prove Property F for $ \mathcal{Z(C_B)}$. Since $FPdim(\mathcal{Z(C_B)})$ are of all the form $p^{a}q^b$, where $p, q$ are primes and $a, b$ are nonnegative integers. Thus all $ \mathcal{Z(C_B)}$ are weakly group-theoretical and therefore have Property F \cite{GN19}. 
\end{proof}

\section{Realization via quantum abelian Chern-simons theories}\label{conkmatrix}

Quantum abelian Chern-Simons (CS) theory has been used to model abelian fractional quantum Hall states.  In physical literature, such a theory is indexed, not necessarily unique, by a $K$-matrix, i.e., a non-singular  integral symmetric matrix.  Our interest is only for boson/spin theories, so the $K$-matrix has to be even.  Abelian CS theories are gauge theories with torus gauge group $T^n$, so the $K$-matrix simply represents the level in $H^4(BG,\mbbZ)$ of CS theories.

Given an abelian anyon model $\mcA=(A,\theta)$, $\mcA$ is said to be realized by a CS theory if there is an even $K$-matrix $K_A$ such that the associated anyon model is $\mcA$ as described below.

\subsection{Lattice abelian anyon models}

A lattice $\Lambda$ is a free $\mbbZ$-module of finite rank=$n$ with an integral non-singular symmetric bilinear form $<x,y>: \Lambda \times \Lambda \rightarrow \mbbZ$.  Its dual lattice is the module $\Lambda^*=\textrm{Hom}(\Lambda, \mbbZ)$, which is not an integral lattice unless $\Lambda$ is unimodular ($|\textrm{det}\Lambda|=1$) or self-dual ($\Lambda\cong \Lambda^*$), which are equivalent in our case.  There is an obvious inclusion of $\Lambda \subset \Lambda^*$ and {\it the discriminant group} $A_\Lambda$ of $\Lambda$ is the finite abelian group $\Lambda^*/\Lambda$ of order=$|\textrm{det}\Lambda|$.  

There is a well-known subtlety between quadratic forms on finite abelian groups to $\mbbQ/\mbbZ$ and even lattices. In order for a quadratic form to be able to be always lifted to a bilinear form over $\mbbQ/\mbbZ$, we follow the use of quadratic forms to $\mbbQ/2\mbbZ$ for discriminant forms of even lattices below, denoted as $q_{2,\Lambda}$ for a lattice $\Lambda$.

The bilinear form $<,>$ on $\Lambda$ extends to bilinear from $b(,)$ on $\Lambda^*$, which descends to a $\mbbQ/\mbbZ$-valued bilinear form on $A_\Lambda$ by $b([x],[y])=<x,y>$ mod $\mbbZ$.  If $\Lambda$ is even, i.e. $b(x,x)$ is even for any $x\in \Lambda$, then $\Lambda$ induces a $\mbbQ/2\mbbZ$-valued quadratic form $q_{2,\Lambda}$ on $A_\Lambda$, called {\it the discriminant form}, by:

$$2b([x],[y])=q_{2,\Lambda}([x+y])-q_{2,\Lambda}([x])-q_{2,\Lambda}([y]), q_{2,\Lambda}([0])=0.$$
Then the metric group $(A_\Lambda, q_{2,\Lambda})$ determines an abelian anyon model $(A_\Lambda, \theta)$ with $\theta ([x])=e^{\pi i q_{2,\Lambda}([x])}$, and will be called the associate abelian anyon model to the even lattice $\Lambda$.

If a basis of $\Lambda$ is chosen so that $\Lambda\cong \mbbZ^n$, then its Gram matrix $K_\Lambda$ is a $K$-matrix $K_A$.

In the sequel, we will use $K_A$ for a $K$-matrix for an abelian anyon model $(A,\theta)$ and $W_A$ the inverse matrix of $K_A$, which is the Gram matrix of the dual lattice $L_A^*$.  The lattice with $K_A$ as Gram matrix will be denoted by $L_A$.  

If $K_A$ is regarded as a map $K_A: \mbbZ^n\rightarrow \mbbZ^n$, then the discriminant group of $L_A$ is isomorphic to cokernel group  $\mbbZ^n/K_A(\mbbZ^n)$.  The central charge of the associated anyon model of $L_A$ is the signature of $K_A$ and the order of the discriminant group is $|\textrm{det}(K_A)|$.

If $K_A$ is $+$-definite, then there is an orthogonal matrix $O$ such that $K_A=O^tDO$, where $D$ is diagonal with positive diagonal entries.  Then the row vectors $\{e_i\}$ of $O^t$ are eigen-vectors of $K_A$ and the lattice $L_A$ is generated by the row vectors of $O^t\sqrt{D}$, which is a generating matrix of the lattice $L_A$. 

\subsection{Wall algorithm}

The following algorithm is used to construct bilinear forms on finite abelian groups in \cite{Wall63}.  We adapt the algorithm to find $K$-matrices for abelian anyon models explicitly, first not necessarily $+$-definite in this section and then $+$-definite ones later for VOAs.

The input of the algorithm is a prime power $p^r$ for some positive integer $r$ and a natural number $n$ such that $(n,p)=1$ and $0<n<p^r$\footnote{In \cite{Wall63}, the integer $n$ is the value in the bilinear from $b(x,x)=\frac{n}{p^r}$ for a generator $x$ of a cyclic prime abelian anyon model.  In the applications below, we are mainly interested in the case that $n$ is actually the central charge.  It is not always true that we will obtain the right central charge or discriminant form from the output of the algorithm.}.  The output of the algorithm is a tri-diagonal matrix $W_A$ whose inverse $K_A$ is an even integral matrix with $|\textrm{det}(K_A)|=p^r$, not necessarily definite.

Since $(n,p^r)=1$, then the integral equation $1=nd_1-p^rd_2$ has solutions for some $d_1$ and $d_2$.

The original Wall algorithm is to find integers $\{d_i,a_i\}, 1\leq i\leq k$ in $k$ steps as follows:
\begin{align*}\label{1}
1 &= nd_1-p^rd_2\\
  d_1 &=a_1d_2-d_3 \\
  \vdots\\
  d_{k-1}&=a_{k-1}d_k-d_{k+1}\\
  d_k &=a_kd_{k+1}\\
where  \quad d_{k+1}&=\pm 1.
\end{align*}

For each i, $a_{i-1}d_i$ is taken to be the closest even multiple of $d_i$ to $d_{i-1}$, and the remainder $d_{i+1}$
would satisfy $\left|d_{i+1}\right|<\left|d_i\right|$. The algorithm terminates when the remainder $d_{k+2}$ is 0, then $d_{k+1}$ would be a common divisor of $d_1$ and $d_2$, which forces $d_{k+1}=\pm 1.$

The output $(k+1)\times (k+1)$-matrix $W_A$ of the algorithm is:
\begin{equation}W_A=
\left(
\begin{array}{cccccc}
np^{-r} & 1 & 0 & \cdots & 0 & 0 \\
1 & a_1 & 1 &  \cdots & 0& 0 \\
0 & 1 & a_2 & \cdots & 0 & 0 \\
\vdots &  \vdots &  \vdots &  \ddots &  \vdots &  \vdots \\
0 & 0 & 0 & \cdots & a_{k-1} & 1 \\
0 & 0 & 0 &\cdots & 1& a_k 
\end{array}
\right)
\end{equation}

We denote the determinant of the last $i$ rows and columns by $A_i$.
Suppose $d_{k+1}=\epsilon=\pm 1$, then $A_1=a_k= \epsilon d_k,A_2=a_{k-1}a_k-1=\epsilon a_{k-1}d_k-1= \epsilon d_{k-1}.$
Inductively, suppose $A_i= \epsilon d_{k+1-i}$, then $A_{i+1}=a_{k-i}A_i-A_{i-1}=a_{k-i}(\epsilon d_{k+1-i})-\epsilon d_{k-i+2}=\epsilon d_{k-i}$.
Then the determinant of $K_A$ is $np^{-r}A_k-A_{k-1}=np^{-r}(\epsilon d_1)-(\epsilon d_2)=\epsilon p^{-r}$.  In either case, $\left|det(K_A)\right|=\left|A\right|$ as $K_A=W_A^{-1}$.

\subsection{Construction of K-matrices for abelian anyon models}

The prime decomposition of abelian anyon models 
reduces the construction of $K$-matrices to the prime ones.  
In this section, we construct a $K$-matrix
for each prime abelian anyon model using the Wall algorithm.

Let $p$ be an odd number, and let $c=2m$ or $2n$ for $A_{p^r}$ or 
$B_{p^r}$ as in Thm. \ref{classification_abelian}, respectively, excluding the case $A_3$, which can be realized by $E_6$ or its complement. 
Note that $(c, p)=1$ and $1\leq c<p$.
It follows that $1=cd_1-p^rd_2$ has integral solutions for $d_1$ and $d_2$. 

Suppose $d_1^{'}$ and $d_2^{'}$ are one particular solution of the equation $1=cd_1-p^rd_2$, then the general solution of $1=cd_1-p^rd_2$ would be $d_1=d_1^{'}+p^{r}s,d_2=d_2^{'}+cs$,
where $s$ is a parameter integer. Suppose $0<c<p^r$, we can choose a large s such that $d_1>0$ and $d_2>0$, then it follows that $0<d_2<d_1$.  Since $p^r$ is odd, $c$ is even, then $d_2$ would be odd as $d_1$ can always be chosen to be even by a shift of $p^r$.  Note that in Wall algorithm, $d_s$ has the same parity as $d_{s+2}$ for $1\leq s\leq k-1$, so in this case, $k$ must be odd.

Parallel to above discussion, set $c=1$ for $A_{2^r}, r\geq1$; set $c=3$ for $D_{2^r}, r\geq2$; set $c=5$ for $C_{2^r}, r\geq 3$;  set $c=7$ for $B_{2^r}, r\geq 3$.  We always have $(c, 2)=1$ and $1\leq c<2^r$.
Suppose Wall's algorithm terminates with $d_{k+1}=\pm 1$, then  $k$ must be even.

We will denote the resulting $K$ matrices as $K_{A_{p^r}}, K_{B_{p^r}}, K_{A_{2^r}}, K_{B_{2^r}}, K_{C_{2^r}}, K_{D_{2^r}}$ according to the types of $c$ and $p$.

As seen in last section, the determinant of $K_A$ is as desired. Lem \ref{signature} and Cor \ref{generalized signature} below will be used to verify that the signature of $K_A$ equals to the central charge of $A$ for an explicit $K_A$.
 
\begin{lem}[Sylvester 1852, Gundelfinger 1881, Frobenius 1895]\label{signature}

The eigenvalues $\lambda_k(S)$ of a regular symmetric $n\times n$ matrix $S$ have the signs of the successive minor quotients
\[sign(\lambda_k(S))=sign(\mu_k(S)/\mu_{k-1}(S))\in \{-1,1\}\]
for $k=1, 2, \cdots, n$, with $\mu_0(S)=1.$ The signature is 
\[\tau(S)=\sum_{k=1}^{n}sign(\mu_k(S)/\mu_{k-1}(S))\in\{-n, -n+1, \cdots, n\} \]
\end{lem}
\begin{cor}\label{generalized signature}
If $S$ is an invertible symmetric $n\times n$ matrix which is not regular then for sufficiently small $\epsilon\neq 0 $ the symmetric 
$n\times n$ matrix $S_{\epsilon}=S+\epsilon I_n$ is regular, with eigenvalues $\lambda_k(S_{\epsilon})=\lambda_k(S)+\epsilon\neq 0,$ 
and  \[\begin{split}
&sign(\lambda_k(S_\epsilon))=sign(\lambda_k(S))\in \{-1,1\},\\
&\tau(S)=\tau(S_{\epsilon})=\sum_{k=1}^{n}sign(\mu_k(S_{\epsilon})/\mu_{k-1}(S_{\epsilon}))\in \mathbb{Z}.
\end{split}\]

\end{cor}

\begin{thm}\label{K Wall}

For each of the six $A,B,C,D$ families of abelian anyon models $(A,q)$, an even $K$-matrix $K_A$ can be constructed such that $\mathbb{Z}^{k+1}/K_A(\mathbb{Z}^{k+1})\cong A$, where Wall's algorithm terminates with $d_{k+1}=\pm 1$ and the signature of $K_A$ equals to the central charge of $A$.

\begin{enumerate}
    \item\label{existence of K, p odd}
    For abelian anyon models $A=A_{p^r}$ , $B_{p^r}$, the choices of $2m$ and $2n$ as above lead to even $K$-matrices $K_A$ that realize the abelian anyon models $(A,q)$.
    \item\label{existence of K, 2  even}
For abelian anyon models $A=A_{2^r}$ , $B_{2^r}$, $C_{2^r}$, $D_{2^r}$, the choices of $c=1,3,5,7$ as above lead to even $K$-matrices that realize the abelian anyon models $(A,q)$.
\item\label{existence of E and F}
For abelian anyon models $A=E_{2^r}$ , $F_{2^r}$,  even $K$-matrices $K_A$ can be directly constructed such that $\mathbb{Z}^{2}/K_{E_{2^r}}(\mathbb{Z}^{2})\cong E_{2^r}$ and 
$\mathbb{Z}^{4}/K_{F_{2^r}}(\mathbb{Z}^{4})\cong F_{2^r}$, the signature of $K_A$ equals the central charge of $A$ that realize the corresponding anyon models.

\end{enumerate}

\end{thm}

\begin{proof} 
To prove Part \ref{existence of K, p odd} of Thm. \ref{K Wall}, suppose  Wall's algorithm terminates with $d_{k+1}=s$ where $s=\pm 1$, then $det(W_A)=s p^{-r}, det(K_A)=s p^r$ and $k$ is odd.

First, we explain that $K_A$ is an even matrix. As $K_A=W_{A}^{-1}=s p^rW_{A}^{*}$,  all elements in $W_{A}^{*}$ are   integers or at most a denominator of $p^r$ coming from the element $\frac{c}{p^r}$, thus $K_A$ is an integral matrix. 

Set $A_t$ be the submatrix of $W_A$ by deleting the last k+1-t rows and the last k+1-t columns, $f_1(t)$ is its determinant, $f(t)=s p^{r}f_1(t)$. Set   $D_{k+1-t}$ be the submatrix of $W_A$ by deleting the first $t$ rows and the first $t$ columns, $g(k+1-t)=d_t$ is its determinant.

$(K_A)_{1,1}=s p^r(W_A^{*})_{1,1}=s p^{r}d_1$;

$(K_A)_{k+1,k+1}=s p^{r}(W_A^{*})_{k+1,k+1}=s p^{r}f_1(k)=f(k)$;

 $2\leq t \leq k,$
 \[ (K_A)_{t, t}=s p^{r}(W_A^{*})_{t, t}=s p^{r}f_1(t-1)g(k+1-t)=d_{t}f(t-1)\]
 As $f_1(t+1)=a_{t}f_1(t)-f_1(t-1)$, so $f(t+1)=a_{t}f(t)-f(t-1)$ which means $f(t+1)$ has the same parity as $f(t-1)$, as $f(1)=s c$ is even, $f(2)=s p^{r}(\frac{c a_1}{p^r}-1)=
 s(c a_1-p^r)$ is odd, so \[f(t)=\begin{cases} \text{even} &t \text{ odd } \\ \text{odd}&t \text{ even}
\end{cases}\]

Besides, \[d_t=\begin{cases} \text{even} &t \text{ odd } \\ \text{odd}&t \text{ even}
\end{cases}\]

So $(K_A)_{t, t}=d_{t}f(t-1)$ is always even.
As $d_1$ is even, so $(K_A)_{1,1}$ is even; as $k$ is odd, so 
$(K_A)_{k+1,k+1}=f(k)$ is even.
Thus all diagonal elements in $K_A$ are all even which means $K_A$ is an even matrix.
 
Next, we prove $\mathbb{Z}_{k+1}/K_A\mathbb{Z}_{k+1}\cong \mathbb{Z}_{p^r}.$  $K_AK_A^{*}=det(K_A)E_{k+1}=s p^rE_{k+1},$ so \[K_A^{*}=s p^rK_A^{-1}=s p^rW_A=\left(
\begin{array}{cccccc}
s c & s p^r &0& \cdots & 0 &0\\
s p^r &s p^ra_1 & s p^r & \cdots & 0 &0 \\
0 & s p^r & s p^r a_2& \cdots& 0 &0 \\
\vdots & \vdots & \vdots & \ddots & 0 & 0 \\
0 & 0 & 0 &\cdots & sp^ra_k&s p^r \\
0 & 0 & 0 & \cdots & s p^r & s p^ra_{k+1}\\
\end{array}
\right)\]
 
We consider the invariant factor of $K_A$, set $\alpha_t$ be the $t$ th
invariant factor of $K_A$, $\alpha_t=\frac{d_t(K_A)}{d_{t-1}(K_A)}$.
where $d_t$ called $t$-th determinant divisor equals the greatest common divisor of all $t\times t$ minors of the matrix $K_A$ and $d_0:=1.$

$(K_A^{*})_{1,1}=s c, (K_A^{*})_{1, 2}=s p^r$ which means $d_{k} \,|\,(s c, s p^r)=1$, so $d_k=1$  which means $\alpha_t=1$ for $1\leq t\leq k$ and $\alpha_{k+1}=p^r$. Then the Smith normal form of $K_A$ verifies $\mathbb{Z}_{k+1}/K_A\mathbb{Z}_{k+1}\cong \mathbb{Z}_{p^r}.$ 

Finally, we show the signature of $K_A$ equals the central charge of A by giving explicit formula of $K_A$ and calculate its signature according to  Lem\ref{signature} and Cor\ref{generalized signature}, comparing with central charge of  Thm\ref{ calculate c}.

For $A_3,$ we take $K_{A_3}=K'(3,2)^{\perp}$, see Thm\ref{exstence KA}.

For $A_{p^r}$ except $A_3$, we can always choose $m=2$, then $c=2 m=4.$ Using Wall's algorithm, 

When $r$ is even, 
\[K_{A_{p^r}}=\begin{pmatrix}\frac{4}{p^r}&1&0&0\\
1&\frac{p^r-1}{4}&1&0\\
0&1&-4&1\\
0&0&1&-2 \\
\end{pmatrix}^{-1}\]

$K_{A_{p^r}}$ has signature 0 which is consistent with central charge 0(mod \, 8).

When $r$ is odd and $p\equiv1(mod\,8)$ ,
\[K_{A_{p^r}}=\begin{pmatrix}\frac{4}{p^r}&1&0&0\\
1&\frac{p^r-1}{4}&1&0\\
0&1&-4&1\\
0&0&1&-2 \\
\end{pmatrix}^{-1}\]

$K_{A_{p^r}}$ has signature 0 which is consistent with central charge 0(mod \, 8).

When $r$ is odd and $p\equiv-3(mod\,8)$ ,
\[K_{A_{p^r}}=\begin{pmatrix}\frac{4}{p^r}&1&0&0\\
1&\frac{p^r+3}{4}&1&0\\
0&1&2&1\\
0&0&1&2 \\
\end{pmatrix}^{-1}\]

$K_{A_{p^r}}$ has signature 4 which is consistent with central charge 4(mod \, 8) and it is positive definite.

When $r$ is odd and $p\equiv-1(mod\,8)$ ,
\[K_{A_{p^r}}=\begin{pmatrix}\frac{4}{p^r}&1\\
1&\frac{p^r+1}{4}\\
\end{pmatrix}^{-1}\]

$K_{A_{p^r}}$ has signature 2 which is consistent with central charge 2(mod \, 8) and it is positive definite.

When $r$ is odd and $p\equiv3(mod\,8)$ ,
\[K_{A_{p^r}}=\begin{pmatrix}\frac{4}{p^r}&1&0&0\\
1&\frac{p^r-3}{4}&1&0\\
0&1&-2&1\\
0&0&1&-2 \\
\end{pmatrix}^{-1}\]

$K_{A_{p^r}}$ has signature -2 which is consistent with central charge 6(mod \, 8).

For $B_{p^r}, r$ even and  $p\equiv-3(mod\,8)$ we can  choose $m=1$, then $c=2 m=2.$ 

\[K_{B_{p^r}}=\begin{pmatrix}\frac{2}{p^r}&1&0&0\\
1&\frac{p^r-1}{2}&1&0\\
0&1&-2&1\\
0&0&1&-2 \\
\end{pmatrix}^{-1}\]

$K_{B_{p^r}}$ has signature 0 which is consistent with central charge 0(mod \, 8).

For $B_{p^r}, r$ even and  $p\equiv-1(mod\,8)$ we can  choose $m=\frac{p-1}{2}$, then $c=2 m=p-1.$ 

\[K_{B_{p^r}}=\begin{pmatrix}\frac{p-1}{p^r}&1&0&0\\
1&\frac{p^r-1}{p-1}&1&0\\
0&1&-2&1\\
0&0&1&-2 \\
\end{pmatrix}^{-1}\]

$K_{B_{p^r}}$ has signature 0 which is consistent with central charge 0(mod \, 8).

For $B_{p^r}, r$ even and  $p\equiv3(mod\,8)$ we can  choose $m=1$, then $c=2 m=2.$ 

\[K_{B_{p^r}}=\begin{pmatrix}\frac{2}{p^r}&1&0&0\\
1&\frac{p^r-1}{2}&1&0\\
0&1&-2&1\\
0&0&1&-2 \\
\end{pmatrix}^{-1}\]

$K_{B_{p^r}}$ has signature 0 which is consistent with central charge 0(mod \, 8).

For $B_{p^r}, r$ odd and  $p\equiv-3(mod\,8)$ we can  choose $m=1$, then $c=2 m=2.$ 

\[K_{B_{p^r}}=\begin{pmatrix}\frac{2}{p^r}&1&0&0\\
1&\frac{p^r-1}{2}&1&0\\
0&1&-2&1\\
0&0&1&-2 \\
\end{pmatrix}^{-1}\]

$K_{B_{p^r}}$ has signature 0 which is consistent with central charge 0(mod \, 8).

For $B_{p^r}, r$ odd and  $p\equiv-1(mod\,8)$ we can  choose $m=\frac{p-1}{2}$, then $c=2 m=p-1.$ 

\[K_{B_{p^r}}={\begin{pmatrix} \frac{p-1}{p^r} & 1 &0&\cdots&0&0\\ 1 &  \frac{p^r-1}{p-1}+1&1&\cdots&0&0\\
 0&1&2&\cdots&1&0\\\vdots&\vdots&\vdots&\ddots&\vdots&\vdots\\
 0&0&0&\cdots&2&1\\0&0&0&\cdots&1&2 \end{pmatrix}}^{-1}\]

$K_{B_{p^r}}$ is $(p-1)\times (p-1)$ has signature $p-1$ which is consistent with central charge 6(mod \, 8).

For $B_{p^r}, r$ odd and  $p\equiv3(mod\,8)$ we can  choose $m=1$, then $c=2 m=2.$ 

\[K_{B_{p^r}}=\begin{pmatrix}\frac{2}{p^r}&1\\
1&\frac{p^r+1}{2}\\
\end{pmatrix}^{-1}\]

$K_{B_{p^r}}$ has signature 2 which is consistent with central charge 2(mod \, 8) and it is positive definite.

The proof of Part \ref{existence of K, 2 even} of Thm. \ref{K Wall} essentially follows the same steps above by observing that in this case $k$ is even: 
\[f(t)=\begin{cases} \text{odd} &t \text{ odd } \\ \text{even}&t \text{ even}
\end{cases}\]

\[d_t=\begin{cases} \text{odd} &t \text{ odd } \\ \text{even}&t \text{ even}
\end{cases}\]
and some  calculations  may be different but do not affect the  results. 

Finally, we show the signature of $K_A$ equals the central charge of A by giving explicit formula of $K_A$ and calculate its signature according to  Lem. \ref{signature} and Cor. \ref{generalized signature}, comparing with central charge of Thm. \ref{ calculate c}.

For $A_{2^r}$, $c=1,$ using Wall's algorithm, 

\[K_{A_{2^r}}=\begin{pmatrix}\frac{1}{2^r}&1&0\\
1&2^r&1\\
0&1&-2\\
\end{pmatrix}^{-1}\]

$K_{A_{2^r}}$ has signature 1 which is consistent with central charge 1(mod \, 8).

For $D_{2^r}, r$ even, $c=3,$ 

\[K_{D_{2^r}}=\begin{pmatrix}\frac{3}{2^r}&1&0\\
1&\frac{2^r+2}{3}&1\\
0&1&2\\
\end{pmatrix}^{-1}\]

$K_{D_{2^r}}$ has signature 3 which is consistent with central charge 3(mod\, 8) and it is positive definite.

For $D_{2^r}, r$ odd, $c=3,$ 

\[K_{D_{2^r}}=\begin{pmatrix}\frac{3}{2^r}&1&0&0&0\\
1&\frac{2^r-2}{3}&1&0&0\\
0&1&-2&1&0\\
0&0&1&-2&1\\
0&0&0&1&-2
\end{pmatrix}^{-1}\]

$K_{D_{2^r}}$ has signature -1 which is consistent with central charge 7(mod\, 8).

For $C_{2^2},$ we take $K_{C_{2^2}}=K_e(2)^{\perp}$, see Thm\ref{exstence KA}.

For $C_{2^r}, r\equiv1(mod\,4)$ , $c=5,$ 

\[K_{C_{2^r}}=\begin{pmatrix}\frac{5}{2^r}&1&0\\
1&\frac{2^r-2}{5}&1\\
0&1&-2\\
\end{pmatrix}^{-1}\]

$K_{C_{2^r}}$ has signature 1 which is consistent with central charge 1(mod\, 8).

For $C_{2^r}, r\equiv3(mod\,4)$ , $c=5,$ 

\[K_{C_{2^r}}=\begin{pmatrix}\frac{5}{2^r}&1&0&0&0\\
1&\frac{2^r+2}{3}&1&0&0\\
0&1&2&1&0\\
0&0&1&-2&1\\
0&0&0&1&-2
\end{pmatrix}^{-1}\]

$K_{C_{2^r}}$ has signature 1 which is consistent with central charge 1(mod\, 8).

For $C_{2^r}, r\equiv2(mod\,4),c>2$ , $c=5,$ 

\[K_{C_{2^r}}=\begin{pmatrix}\frac{5}{2^r}&1&0&0&0&0&0\\
1&\frac{2^r-4}{3}&1&0&0&0&0\\
0&1&-2&1&0&0&0\\
0&0&1&-2&1&0&0\\
0&0&0&1&-2&1&0\\
0&0&0&0&1&-2&1\\
0&0&0&0&0&1&-2\\

\end{pmatrix}^{-1}\]

$K_{C_{2^r}}$ has signature -3 which is consistent with central charge 5(mod\, 8).

For $C_{2^r}, r\equiv0(mod\,4)$ , $c=5,$ 

\[K_{C_{2^r}}=\begin{pmatrix}\frac{5}{2^r}&1&0&0&0\\
1&\frac{2^r+4}{3}&1&0&0\\
0&1&2&1&0\\
0&0&1&2&1\\
0&0&0&1&2
\end{pmatrix}^{-1}\]

$K_{C_{2^r}}$ has signature 5 which is consistent with central charge 5(mod\, 8).

For $B_{2},$ we take $K_{B_{2}}=(2)^{\perp}$, see Thm\ref{exstence KA}.

For $B_{2^2},$ we take $K_{B_{2^2}}=(2^2)^{\perp}$, see Thm\ref{exstence KA}.
 
 When $r\geq 3,$

For $B_{2^r}, r\equiv1(mod\,3)$ , $c=7,$ 

\[K_{B_{2^r}}=\begin{pmatrix}\frac{7}{2^r}&1&0&0&0&0&0\\
1&\frac{2^r+6}{3}&1&0&0&0&0\\
0&1&2&1&0&0&0\\
0&0&1&2&1&0&0\\
0&0&0&1&2&1&0\\
0&0&0&0&1&2&1\\
0&0&0&0&0&1&2\\

\end{pmatrix}^{-1}\]

$K_{B_{2^r}}$ has signature 7 which is consistent with central charge 7(mod\, 8).

For $B_{2^r}, r\equiv2(mod\,3)$ , $c=7,$ 

\[K_{B_{2^r}}=\begin{pmatrix}\frac{7}{2^r}&1&0&0&0\\
1&\frac{2^r-4}{3}&1&0&0\\
0&1&-2&1&0\\
0&0&1&-4&1\\
0&0&0&1&-2
\end{pmatrix}^{-1}\]

$K_{B_{2^r}}$ has signature -1 which is consistent with central charge 7(mod\, 8).

For $B_{2^r}, r\equiv0(mod\,3)$ , $c=7,$ 

\[K_{B_{2^r}}=\begin{pmatrix}\frac{7}{2^r}&1&0&0&0&0&0\\
1&\frac{2^r+6}{3}&1&0&0&0&0\\
0&1&2&1&0&0&0\\
0&0&1&2&1&0&0\\
0&0&0&1&2&1&0\\
0&0&0&0&1&2&1\\
0&0&0&0&0&1&2\\

\end{pmatrix}^{-1}\]

$K_{B_{2^r}}$ has signature 7 which is consistent with central charge 7(mod\, 8).

To prove Part \ref{existence of E and F} of Thm. \ref{K Wall},
for $E_{2^r}$, set $K_{E_{2^r}}=\begin{pmatrix}0&2^{-r}\\
2^{-r}&0\end{pmatrix}^{-1}=\begin{pmatrix}0&2^r\\
2^r&0\end{pmatrix}$, then it is easy to verify that 
$\mathbb{Z}^{2}/K_{E_{2^r}}(\mathbb{Z}^{2})\cong E_{2^r}$.
 $K_{E_{2^r}}$ has signature 0 which is consistent with central charge 0(mod\, 8).
 
For $F_{2^r}$, set \[K_{F_{2^r}}=\begin{pmatrix}2^{1-r}&2^{-r}&0&0\\
2^{-r}&2^{1-r}&1&0\\
0&1&2 a&1\\
0&0&1&2 b\\
\end{pmatrix}^{-1}\]
where $a=\frac{2^r-(-1)^r}{3}, b=(-1)^{r-1}.$
   
When $r$ is odd, 
   
\[K_{F_{2^r}}=\begin{pmatrix}\frac{2^{1+r}(1+2^r)}{3}&
\frac{-2^{r}(1+2^{2+r})}{3}&2^{1+r}&-2^r\\
\frac{-2^{r}(1+2^{2+r})}{3}&\frac{2^{1+r}(1+2^{2+r})}{3}&-2^{2+r}&2^{1+r}\\
2^{1+r}&-2^{2+r}&6&-3\\
 - 2^{r}&2^{1+r}&-3&2 \\
\end{pmatrix}\] 
  
When $r$ is even,
  
\[K_{F_{2^r}}=\begin{pmatrix}\frac{2^{1+r}(-1+2^r)}{3}&
\frac{2^{r}(-1+2^{2+r})}{3}&-2^{1+r}&-2^r\\
\frac{2^{r}(-1+2^{2+r})}{3}&\frac{-2^{1+r}(-1+2^{2+r})}{3}&2^{2+r}&2^{1+r}\\
  -2^{1+r}&2^{2+r}&-6&-3\\
 - 2^{r}&2^{1+r}&-3&-2 \\
\end{pmatrix}\] 
  
It is easy to verify that $\mathbb{Z}^{4}/K_{F_{2^r}}(\mathbb{Z}^{4})\cong F_{2^r}$,just consider the third invariant factor of $K_{F_{2^r}}$ and then use Smith normal form as in Thm\ref{existence of K, p odd}. When $r$ is odd, $K_{F_{2^r}}$ has signature 4 which is consistent with central charge 4(mod\, 8); When $r$ is even, $K_{F_{2^r}}$ has signature 0 which is consistent with central charge 0(mod\, 8).

\end{proof}

\section{Realization via chiral conformal field theories}\label{conlattice}

A chiral conformal field theory ($\chi$CFT) is the chiral part of a full conformal field theory, and mathematically is either a vertex operator algebra (VOA) or a local conformal net (LCN).  In this paper, by a $\chi$CFT we will mean a VOA.  In this section, we will investigate two different approaches to realize abelian anyon models by VOAs with small central charges explicitly: the complement lattice approach and the Wall algorithm approach.  The complement lattice approach was used in \cite{EG20} to realize all abelian anyons models by lattice VOAs.  We will make the complement lattice approach more algorithmic.  The Wall algorithm is not as powerful as the complement lattice approach, but constructive and much more elementary.  For applications to the bulk-edge correspondence in abelian fractional quantum Hall states, we are interested in explicit realizations of small central charges.

\subsection{Lattice VOAs}

Given an even $+$-definite lattice $\Lambda$ of rank=$c$, there associates a VOA $\mcV_\Lambda$ corresponding physically to $c$ free bosons compactified so that their momenta live in the lattice $\Lambda$.  The vector space $V_\Lambda$ of $\mcV_\Lambda$ is $V_\Lambda=V^{\otimes c}\otimes \mbbC[\Lambda]$, where $V=\mbbC[x_1,x_2,\cdots]$ is the polynomial algebra $\mbbC[x_i]$ of $x_i$ with deg$(x_i)=i$, and $\mbbC[\Lambda]$ the group algebra of $\Lambda$.

An abelian anyon model $(A,q)$ is said to be realized by a lattice VOA $\mcV_{\Lambda}$ or simply lattice $\Lambda$ for some even $+$-definite lattice $\Lambda$ if the representation category of the associated VOA is $(A,q)$.  The central charges $c_\mcV$ and conformal weights $\{h_a\}$ of a nontrivial lattice VOA $\mcV_{\Lambda}$ are positive rational numbers that lifts the central charge $c_A$ of the abelian anyon model and the exponents of the topological twists $\theta_a=e^{2\pi h_a}$.  The genus of a VOA is the pair $((A,q),c_\mcV)$.  

It is known that the representation categories of such lattice VOAs are the abelian anyon models associated to the lattice as in Sec. \ref{conkmatrix}.  Therefore, the reconstruction of lattice VOAs is reduced to the construction of even $+$-definite integral lattices. 

\subsubsection{Minimal genus realization}

Given an anyon model $\mcB$, a candidate genus is a pair $(\mcB, c)$ such that $c$ is a positive rational number $c=c_\mcB $ mod $8$.  The minimal genus of an anyon model $\mcB$ is the genus of a VOA with the smallest realizable central $c$.  Such minimal realizations are not unique as the trivial anyon model shows.  The non-uniqueness of minimal realization is a general phenomenon.

For the eight families of abelian anyon models, if the central charge $0$ mod $8$ is lifted to $8$ in Table \ref{ calculate c} and the rest is taken as in the table, then they all can be realized.  A proof follows from Corollary 1.10. 2 of \cite{Nik1}.  A quadratic form of an abelian anyon models $(A,q)$ can be realized by an even $+$-definite lattice $L$ of rank=$l$ if the central charge of $q$ equals to the rank of $l$ mod $8$, and the rank of $L$ $>1$ for the $A,B,C,D$ families and $>2$ for the $E,F$ families.
Therefore the minimal genus realizations for all eight families do exist, though no known explicit realizations are known to the best knowledge of the authors.

The representation category of a lattice VOA for a unimodular or self-dual even $+$-definite lattice is trivial, i.e. $\mcV ec$.  Therefore, if a lattice $\Lambda$ realizes an abelian anyon model $(A,q)$, then the direct sum $\Lambda \oplus L$ realizes $(A,q)$ as well for any unimodular even $+$-definite lattice $L$. Taking stability of edge theory into consideration to rule out some trivial constructions, we define extremal VOAs inside a given genus as in \cite{TW17}.  Given a genus $(\mcB,c)$, a VOA that realizes this genus is extremal if $rc/4+r(r-1)/2 -6\sum_a h_a$ is the smallest positive integer, where $r$ is the rank of $\mcB$ \cite{TW17}.

\subsection{Complement lattice approach}

Anyon models come in conjugate pairs that the topological twists are complex conjugate of each other.  It is often easier to realize one of the pair whose central charge is smaller.  For example for the Semion and anti-Semion pair, the Semion is realized by $SU(2)_1$ by a $K$-matrix $K_A=(2)$, while the anti-Semion with $c=7$ mod $8$ is harder.  It can be realized by $(E_7)_1$ with a $K$-matrix $E_7$.  So one strategy is to realize one of each pair by constructing an explicit lattice, and then taking its complement in a self-dual even lattice to realize the other in the pair as in \cite{EG20}.  We will carry this strategy first in this section in a more detailed elementary and algorithmic way.  The drawback of the complement lattice approach is that the resulting lattice has a relatively big central charge, so does not realize the minimal genus in general.

\subsubsection{Complement lemma}

Let $L$ be an even $+$-definite symmetric lattice of rank=$c$ with Gram matrix $K_L$.  Suppose $c'$ is a positive integer such that $c'\geq c$ and $c+c'=8l$ for some integer $l$.  Then there exists a primitive embedding of $L$ into a self-dual even $+$-definite lattice $\Lambda$ by Corollary 1.12.3 in \cite{Nik1}.  Let $L^{\perp}$ be the complement lattice of $L\subset \Lambda$, then $\Lambda \cong L\oplus L^{\perp}$.  Therefore, $L^{\perp}$ is an even $+$-definite lattice with associated abelian abyon model $(A,-q)$.  But it is not known how to construct the primitive embedding $L\subset \Lambda$ for the minimal $c'$ in general.  So instead in this section, we will construct a much larger $c'$ algorithmically by going through the following lemma.

\begin{lem}\label{conjugate lattice}

\begin{enumerate}

\item If an abelian anyon model $(A,q)$ is realized by an even $+$-positive lattice $L$, then the diagonal embedding of $L$ into $L^{\oplus 8}$ can be extended to a primitive embedding into a unimodular even $+$-positive lattice $UL^{\oplus 8}$. 

\item For any integer $m>0$ and lattice $\Lambda$, the lattice $\Lambda^{\oplus 8m}$ can be self glued to a self-dual lattice $U\Lambda^{\oplus 8m}$.
Furthermore, for any integer $0<l< m$, the lattice $\Lambda^{\oplus l}$ has a primitive embedding in $U\Lambda^{\oplus 8m}$.

\item Let $L$ be an even $+$-definite lattice with Gram matrix $K_A$ and central charge $c$ that realizes an anyon model $(A,q)$. Then there exists an even self-dual $+$-definite lattice  $\Lambda$ obtained from self-gluing $L^{\oplus 8}$ with $L$ as a primitive sublattice.  
Furthermore, the complement lattice $L^{\perp}=\{x\in \Lambda: x\cdot L=0\}$ of $L$ in $\Lambda$ is an even $+$-definite lattice with central charge $7c$ that realizes the abelian anyon model $(A,-q)$.  

The Gram matrix of $L^{\perp}$ will be denoted by $K_A^{\perp}$. 

\end{enumerate}

\end{lem}

From the structure of Witt group, any $8$ copies of an abelian anyon model is a Drinfeld center, hence $L^{\oplus 8}$ can be self-glued to a self-dual lattice.  This is analogous to the construction the $E_7$ lattice from $8$ copies of $A_1$.

\subsubsection{Construction of even $+$-definite lattices for abelian anyon models}

Among the eight families of abelian anyon models, $B_{p^r}$ for $p$ odd and $A_{2^r}$ are easy to be realized as they are simply the WZW models for $SU(p^r)$ and lattice VOAs with Gram matrix $K_A=(2^r)$.  The most difficult ones are the doubles with $c=0$ mod $8$.

To construct other even $+$-definite $K$-matrices, we start with some lemmas.

\begin{lem}\label{ direct construct K}
Suppose $p\equiv-1$ mod $4$, $r\geq1$,  then for the two families with finite abelian group $\mbbZ_{p^r}$ and central charge $c=p^r-1$, an even $+$-definite matrix $K_A$ can be constructed such that $\mathbb{Z}^c/K_A(\mathbb{Z}^c)\cong A.$ 
This $K_A$ gives a $K$-matrix for $B_{p^r}$, and the complement $K_A^{\perp}$ provides a Gram matrix for $A_{p^r}$.
\begin{proof}
The Cartan matrix of $SU(p^r-1)$ can be directly used to realize anyon models $B_{p^r}$: 
\[K'(p^r,p^r-1)\triangleq K_A=\left(
   \begin{array}{cccccc}
     2 &-1 &0& \cdots & 0 &0\\
    -1 &2 & -1 & \cdots & 0 &0 \\
    0 & -1& 2 & \cdots& 0 &0 \\
     \vdots & \vdots & \vdots & \ddots & 0 & 0 \\
     0 & 0 & 0 &\cdots & 2 &-1 \\
     0 & 0 & 0 & \cdots & -1 & 2 \\
   \end{array}
 \right)\]
By induction, the matrix $K_A$ 
is even positive-definite and $det(K_A)=p^r.$ As $(K_A)_{c,1}^*=-1$, so $d_{c-1}=1$ which means $\alpha_k=1$ for $1\leq k\leq c-1$ and $\alpha_c=p^r$. Then 
  the Smith normal form of $K_A$ verifies $ \mathbb{Z}_c/K_A\mathbb{Z}_c\cong\mathbb{Z}_{p^r}$.
\end{proof}
\end{lem}
 \begin{lem}\label{former reduced det}
Suppose A is $m\times m$ matrix, D is $n\times n$ matrix, B is $m\times n$ matrix, C is $n\times m$ matrix,
if A is invetible, then
\[det(\left(
        \begin{array}{cc}
          A & B \\
          C & D \\
        \end{array}
      \right)
)=det(A)det(D-CA^{-1}B)\]
\begin{proof}
Consider the column transformation \[\left(
                                       \begin{array}{cc}
                                         A & B \\
                                         C & D \\
                                       \end{array}
                                     \right)\left(
                                              \begin{array}{cc}
                                                E_m & -A^{-1}B \\
                                               0 & E_n \\
                                              \end{array}
                                            \right)=\left(
                                                      \begin{array}{cc}
                                                        A & 0\\
                                                       C& D-CA^{-1}B\\
                                                      \end{array}
                                                    \right)
\]
Where $E_m$ is identity matrix of order m.

Take the determinant of both sides which completes our proof.
\end{proof}

\end{lem}

\begin{lem}\label{construct  K}
Suppose $p\equiv1$mod4, for $r\geq1$ and $s=\pm1$, $s=1$ correspondents to $A_{p^r}$, $s=-1$ correspondents to $B_{p^r}$, there is a prime $p'\equiv-1$mod4 satisfying $(\frac{2p^r}{p'})=1$ and $(\frac{2p'}{p})=s$, then for group $A=\mathbb{Z}_{p^r}$ and central charge $c=p'+1$, we can always construct an even positive definite $K_A$ and $\mathbb{Z}_c/K_A\mathbb{Z}_c\cong A.$
\begin{proof}
When  $p\equiv1$mod4, suppose $(\frac{2}{p})=m$, where $m=\pm1$, then $(\frac{2p'}{p})=s$ 
is equivalent to $(\frac{p'}{p})=ms$. Consider $1=(\frac{2p^r}{p'})=(\frac{2}{p'})(\frac{p}{p'})^r=(\frac{2}{p'})(\frac{p'}{p})^r=(\frac{2}{p'})(ms)^r$, then $(\frac{2p^r}{p'})=1$ is equivalent to $(\frac{2}{p'})=(ms)^r$. Then the conditions about $p'$ reduced to $(\frac{p'}{p})=ms$ and $(\frac{2}{p'})=(ms)^r$. According to Chinese Residue Theorem and Dirichlet Theorem, there are infinitely many prime $p'$ such that $p'\equiv d$(mod p) and $p'\equiv\pm1,\pm3$(mod 8), where d is a quadratic residual of p if ms=1, d is a quadratic nonresidual of p if ms=-1.
As $(\frac{2p^r}{p'})=1$, so there is t such that $t^2\equiv 2p^r$mod($p'$), where $1\leq t\leq p'$. Set $c=p'+1$, we claim that 
\[K''(p^r,p'+1,s)\triangleq K_A=
 \left(
   \begin{array}{ccccccccc}
     \frac{p'p^r+1}{2} & 0&0 &0& \cdots&0& \cdots & 0 &p^r\\
     0&2 & -1 & 0&\cdots&0& \cdots & 0 &0 \\
    0 & -1 & 2 &-1& \cdots&0& \cdots& 0 &0 \\
     \vdots & \vdots& \vdots & \vdots & \ddots & \vdots & \ddots&\vdots &\vdots  \\
     0 & 0 & 0 &0&\cdots & 2 &\cdots&0& 1 \\
\vdots & \vdots& \vdots & \vdots & \ddots & \vdots &\ddots&\vdots &\vdots \\
0 & 0 & 0 &0&\cdots &0& \cdots& 2 & 0 \\
     p^r & 0 & 0 &0& \cdots &1& \cdots& 0 & \frac{2p^r+t(p'-t)}{p'} \\
   \end{array}
 \right)
 \]
is an even positive definite matrix and $\mathbb{Z}_c/K_A\mathbb{Z}_c\cong A,$ where the order of $K_A$ is c and it has entry 1 in $(c-t,c)$ and $(c,c-t)$ position.

As $p'\equiv-1$mod4 and $p\equiv1$mod4, so $\frac{p'p^r+1}{2}$ is even.
As $t^2\equiv 2p^r$mod($p'$), so $\frac{2p^r+t(p'-t)}{p'}$ is even. Next we show $K_A$ is positive definite by computing its principal minors. 

Set $A_s$ be the submatrix of $K_A$ by deleting the last c-s rows and the last c-s columns, $f(s)$ is its determinant, we need to prove $f(s)>0$ for all $1\leq s\leq c$. $f(1)=\frac{p'p^r+1}{2}>0, f(2)=p'p^r+1>0.$ In fact, when $1\leq s\leq c-1$, we have $f(s)=\frac{(p'p^r+1)s}{2}>0$. Expand $K_A$ along the first column, $f(c)=\frac{p'p^r+1}{2}\times2p^r-p^rp^rp'=p^r>0,$ where $2p^r$ is the determinant of M.
\[M=\left(
   \begin{array}{cccccccc}
     
    2 & -1 & 0&\cdots&0& \cdots & 0 &0 \\
     -1 & 2 &-1& \cdots&0& \cdots& 0 &0 \\
      \vdots& \vdots & \vdots & \ddots & \vdots & \ddots&\vdots &\vdots  \\
     0 & 0 &0&\cdots & 2 &\cdots&0& 1 \\
\vdots& \vdots & \vdots & \ddots & \vdots &\ddots&\vdots &\vdots \\
 0 & 0 &0&\cdots &0& \cdots& 2 & 0 \\
     0 & 0 &0& \cdots &1& \cdots& 0 & \frac{2p^r+t(p'-t)}{p'} \\
   \end{array}
 \right)\]
 Using Lem\ref{former reduced det} and set 
 \[A=\left(
   \begin{array}{ccccccc}
     
    2 & -1 & 0&\cdots&0& \cdots & 0  \\
     -1 & 2 &-1& \cdots&0& \cdots& 0  \\
      \vdots& \vdots & \vdots & \ddots & \vdots & \ddots&\vdots  \\
     0 & 0 &0&\cdots & 2 &\cdots&0 \\
\vdots& \vdots & \vdots & \ddots & \vdots &\ddots&\vdots \\
 0 & 0 &0&\cdots &0& \cdots& 2  \\
   
   \end{array}
 \right)\]
 $B=(0,0,\cdots,1,\cdots,0)^t$ which is a column vector of dimension $p'-1$ and has 1 in $p'-t$ position and 0 otherwise. $C=(0,0,\cdots,1,\cdots,0)$ which is a row vector of dimension $p'-1$ and has 1 in $p'-t$ position and 0 otherwise. $D=(\frac{2p^r+t(p'-t)}{p'})$ is dimension 1.We can calculate the determinant of $M$:
 \begin{align*}
     det(M)&=det(A)det(D-CA^{-1}B)\\
 &=p'(\frac{2p^r+t(p'-t)}{p'}-A^{-1}_{(p'-t,p'-t)})\\
 &=p'(\frac{2p^r+t(p'-t)}{p'}-\frac{A^{*}_{(p'-t,p'-t)}}{p'})\\
  &=p'(\frac{2p^r+t(p'-t)}{p'}-\frac{(p'-t)t}{p'})\\
  &=p'\times\frac{2p^r}{p'}\\
  &=2p^r
 \end{align*}
 
Next, we show $\mathbb{Z}_c/K_A\mathbb{Z}_c\cong A.$
We consider the invariant factor of $K_A$, take the same notations as in
Thm. \ref{K Wall}, $(K_A)_{c,c}^*=\frac{(p'p^r+1)p'}{2},(K_A)_{c,1}^*=p^rp'$ which means $d_{c-1}\,|\,(\frac{(p'p^r+1)p'}{2},p^rp')=p'$ but $d_{c-1}\,|\, p^r$,
 where $p^r$ is the determiant of $K_A$, as $(p,p')=1$, so $d_{c-1}=1$ which means $\alpha_k=1$ for $1\leq k\leq c-1$ and $\alpha_c=p^r$. Then 
  the Smith normal form of $K_A$ verifies $ \mathbb{Z}_c/K_A\mathbb{Z}_c\cong\mathbb{Z}_{p^r}$.

\end{proof}
\end{lem}

We can construct even positive-definite $K_A$ for those groups and central charge c using Thm. \ref{ calculate c}, Lem. \ref{ direct construct K}, Lem. \ref{former reduced det}, and Lem. \ref{conjugate lattice}.
 
\begin{thm}\label{exstence KA}
Explicit even positive-definite $K$-matrices $K_A$ for the $8$ families of abelian anyon models are given as in Table \ref{tableK} below.  More concretely, 
 \begin{table}[!htbp]\label{tableK}
  \centering 
 \caption{$K_A$}  
 \begin{tabular}{|c|c|c|}
\hline  
{}& $p\equiv1$mod4 &$p\equiv-1$mod4\\
\hline  
$A_{p^r}$&$K''(p^r,p'+1,1)$&$K'(p^r,p^r-1)^{\perp}$\\
\hline
$B_{p^r}$&$K''(p^r,p'+1,-1)$&$K'(p^r,p^r-1)$\\
\hline

\end{tabular}

\begin{tabular}{|c|c|}

\hline  
{}&all r\\
\hline  
$A_{2^r}$&$(2^r)$\\
\hline
$B_{2^r}$&$(2^r)^{\perp}$\\
\hline

\end{tabular}

\begin{tabular}{|c|c|c|}
\hline  
{}& r even &r odd\\
\hline  
$C_{2^r}$&$K_e(r)^{\perp}$&$K_o(r)^{\perp}$\\
\hline
$D_{2^r}$&$K_e(r)$&$K_0(r)$\\
\hline

\end{tabular}

\end{table}
 
\begin{itemize}

\item If $p\equiv -1$ mod4, using Lem. \ref{ direct construct K}, we have $K_A=K'(p^r,p^r-1)$ for $B_{p^r}$, using Lem. \ref{conjugate lattice}, we have $K_A=K'(p^r,p^r-1)^{\perp}$ for $A_{p^r}$.

\item If $p\equiv 1$ mod4, using Lem. \ref{construct K}, we have $K_A=K''(p^r,p'+1,1)$ for $A_{p^r}$ when s=1, and  $K_A=K''(p^r,p'+1,-1)$ for $B_{p^r}$ when s=-1.
    \item For $A_{2^r},c=1$ has $K_A=(2^r)$.
    \item For $B_{2^r}$, using Lem. \ref{conjugate lattice}, we have $K_A=(2^r)^{\perp}$ and we can choose c=7.

 \item For $D_{2^r}, r\geq2$, r even, we have $c\equiv3$ mod8. 
    we can always choose c=3, \[K_e(r)=\left(
   \begin{array}{ccc}
     
    \frac{2^r+2}{3} & 0 &1  \\
     0& 2 &-1 \\
   
 1 & -1 &2  \\
   
   \end{array}
 \right)\]
    
   \item For $D_{2^r}, r\geq3,$ r odd, we have $c\equiv7$mod8.   
    we can always choose c=7, \[K_o(r)=\left(
   \begin{array}{ccccccc}
     
    \frac{2^r+4}{3} & 0 &1 &0&0&0&-1 \\
     0& 2 &-1&0&0&0&0 \\
   
 1 & -1 &2&-1&0&0&0  \\
  0&0 & -1 &2&-1&0&-1\\
  0&0&0& -1 &2&-1&0\\
  0&0&0&0& -1 &2&0\\
  -1&0&0& -1&0&0&2\\
   \end{array}
 \right)\]

\item For $C_{2^r}, r\geq2$, using Lem\ref{conjugate lattice}, we have $K_A=(K_{D_{2^r}})^{\perp}$.
    
\end{itemize}
 
\end{thm}

\begin{example}

\begin{itemize}

\item The case $B_{5}, p=5, r=1$ is an interesting example.  As $p \equiv 1$ mod $4, s=-1$, using Lem. \ref{construct  K}, we look for a $p' \equiv -1$ mod 4 satisfying $(\frac{2\times5}{p'})=1$ and $(\frac{2p'}{5})=-1$. Since  $5 \equiv -3$ mod 8, the second condition is equivalent to $(\frac{p'}{5})=1$, we can take $p' \equiv 1$ mod 5, The first one is equivalent to $(\frac{2}{p'})=1$ by quadratic reciprocity. As  $p' \equiv \pm1$ mod 8, thus $p' \equiv -1$  mod8 by considering $p' \equiv -1$ mod 4.  Using the Chinese Remainder Theorem, we know $p' \equiv 31$  mod 40, we can take $p'=31$, so $K''(5,32,-1)$ is the desired matrix.

The size of our solution is not the smallest.  The solution for $c=8$ is unique given by the $K$-matrix in the summary.

The solutions for $c=16$ is completed in \cite{Yang94}.

\item $D_{2^4}$, we know  
     \[K_e(2)=\left(
   \begin{array}{ccc}
     
    6 & 0 &1  \\
     0& 2 &-1 \\
   
 1 & -1 &2  \\
   
   \end{array}
 \right)\]
     is the desired matrix.
     \item $D_{2^3}$, we know  
    \[K_o(3)=\left(
   \begin{array}{ccccccc}
     
    4 & 0 &1 &0&0&0&-1 \\
     0& 2 &-1&0&0&0&0 \\
   
 1 & -1 &2&-1&0&0&0  \\
  0&0 & -1 &2&-1&0&-1\\
  0&0&0& -1 &2&-1&0\\
  0&0&0&0& -1 &2&0\\
  -1&0&0& -1&0&0&2\\
   \end{array}
 \right)\]

is the desired matrix.
     
\end{itemize}

\end{example}

\subsubsection{$K$-matrices for families $E$ and $F$}
Suppose explicit $K_A$ matrices for $B_{2^r}$ and $C_{2^r}$ are given, then we could obtain $K_A$ matrix for $E_{2^r}$ and $F_{2^r}$ from $K_{B_{2^r}}$ and $K_{C_{2^r}}$ as follows.
 
Given $K_A$ for $B_{2^r}$ and $C_{2^r}$ and suppose
$G_{B_{2^r}}$ is the matrix consisting of the row vectors of a basis of $L_{B_{2^r}}$, i.e. 
\[G_{B_{2^r}}=\begin{pmatrix}\alpha_1\\
\alpha_2\\
\vdots\\
\alpha_n\\
\end{pmatrix}\]
$L^{*}_{B_{2^r}}/L_{B_{2^r}}\cong B_{2^r}$, $\alpha_1, \alpha_2, \cdots, \alpha_n$ is a basis of the lattice $L_{B_{2^r}}$, then $G_{B_{2^r}}G_{B_{2^r}}^t=K_{B_{2^r}}$.
  
Similarly, let 
$G_{C_{2^r}}$ be the matrix consisting of the row vectors of a basis of $L_{C_{2^r}}$, i.e. 
 \[G_{C_{2^r}}=\begin{pmatrix}\beta_1\\
\beta_2\\
 \vdots\\
\beta_m\\
  \end{pmatrix}\]
$L^{*}_{C_{2^r}}/L_{C_{2^r}}\cong C_{2^r}$, $\beta_1, \beta_2, \cdots, \beta_m$ is a basis of the lattice $L_{C_{2^r}}$, then $G_{C_{2^r}}G_{C_{2^r}}^t=K_{C_{2^r}}$.

If $K_{B_{2^r}}$ and $K_{C_{2^r}}$ are known, then $G_{B_{2^r}}$ and $G_{C_{2^r}}$ can be constructed. To find $K_A$ matrix for $E_{2^r}$ and $F_{2^r}$, we do the following.

As $L^{*}_{B_{2^r}}/L_{B_{2^r}}\cong B_{2^r}$, there is a generator $\lambda$ of 
$L^{*}_{B_{2^r}}/L_{B_{2^r}}$ satisfying $\lambda \cdot\lambda\equiv \frac{-1}{2^k} $(mod 2) ,
then
\[
G_{E_{2^r}}=\begin{pmatrix}
   \begin{array}{c c| c|c}
   \frac{1}{\sqrt{2^k}}  & \frac{1}{\sqrt{2^k}}   &\lambda  &\lambda \\
   0& \sqrt{2^k}& & \\\hline
   & &\alpha_1& \\
     & &\alpha_2& \\
       & &\vdots& \\
    & &\vdots& \\
    & &\alpha_n& \\ \hline
   &  & &\alpha_1 \\
    &  & &\alpha_2 \\
    & && \vdots \\
    & && \alpha_n \\

   \end{array}
 \end{pmatrix}
\]
As  $L^{*}_{C_{2^r}}/L_{C_{2^r}}\cong C_{2^r}$, there is a generator $\mu$ of 
$L^{*}_{C_{2^r}}/L_{C_{2^r}}$ satisfying $\mu \cdot\mu\equiv \frac{-3}{2^k} $(mod 2) ,
then
\[
G_{F_{2^r}}=\begin{pmatrix}
   \begin{array}{c c c|c}
  \frac{1}{\sqrt{2^k}}  & \frac{1}{\sqrt{2^k}} & \frac{1}{\sqrt{2^k}} &\mu \\
   & \sqrt{2^k}& & \\
  & &\sqrt{2^k}&  \\ \hline
  & & &\beta_1 \\ 
     && &\beta_2 \\
       && &\vdots \\
    && &\vdots \\
    && &\beta_m \\

   \end{array}
 \end{pmatrix}
\]
 
Then  $K_{E_{2^r}}=G_{E_{2^r}}G_{E_{2^r}}^t$ and $K_{F_{2^r}}=G_{F_{2^r}}G_{F_{2^r}}^t$. 

Let us take $r=2$ for example.

$L_{B_{2^2}}=D_7$ and 
\[G_{B_{2^2}}=\begin{pmatrix}\alpha_1\\
\alpha_2\\
\vdots\\
\alpha_7\\
\end{pmatrix}=\begin{pmatrix}-1&-1&0&0&0&0&0\\
1&-1&0&0&0&0&0\\
0&1&-1&0&0&0&0\\
0&0&1&-1&0&0&0\\
0&0&0&1&-1&0&0\\
0&0&0&0&1&-1&0\\
0&0&0&0&0&1&-1\\

\end{pmatrix}\]

$\lambda=(1/2,1/2,1/2,1/2,1/2,1/2,1/2)$ is the generator of $L^{*}_{B_{2^2}}/L_{B_{2^2}}$ satisfying $\lambda \cdot\lambda\equiv \frac{-1}{2^2} $(mod 2) ,then
\[
G_{E_{2^2}}=\begin{pmatrix}
   \begin{array}{c c| c|c}
   \frac{1}{2}  & \frac{1}{2}   &\lambda  &\lambda \\
   0& 2& & \\\hline
   & &\alpha_1& \\
     & &\alpha_2& \\
       & &\vdots& \\
    & &\vdots& \\
    & &\alpha_7& \\ \hline
   &  & &\alpha_1 \\
    &  & &\alpha_2 \\
    & && \vdots \\
    & && \alpha_7 \\

   \end{array}
 \end{pmatrix}
\]

$K_{E_{2^2}}=G_{E_{2^2}}G_{E_{2^2}}^t$=

\[
\addtocounter{MaxMatrixCols}{10}
\begin{pmatrix}
4  &1&-1&{}   &{}&{}&{}   &{}&{}&-1    &{}&{}&{}    &{}&{}&{}\\
1  &4&{}&{}   &{}&{}&{}   &{}&{}&{}     &{}&{}&{}   &{}&{}&{}\\
-1 &{}&2&{}   &-1&{}&{}   &{}&{}&{}     &{}&{}&{}   &{}&{}&{}\\
{}  &{}&{}&2  &-1&{}&{}  &{}&{}&{}     &{}&{}&{}   &{}&{}&{}\\
{}  &{}&-1&-1  &2&-1&{}   &{}&{}&{}     &{}&{}&{}   &{}&{}&{}\\
{}  &{}&{}&{}  &-1&2&-1   &{}&{}&{}     &{}&{}&{}   &{}&{}&{}\\
{}  &{}&{}&{}  &{}&-1&2   &-1&{}&{}     &{}&{}&{}   &{}&{}&{}\\
{}  &{}&{}&{}  &{}&{}&-1   &2&-1&{}     &{}&{}&{}   &{}&{}&{}\\
{}  &{}&{}&{}  &{}&{}&{}   &-1&2&{}     &{}&{}&{}   &{}&{}&{}\\
-1  &{}&{}&{}  &{}&{}&{}   &{}&{}&2     &{}&-1&{}   &{}&{}&{}\\
{}  &{}&{}&{}  &{}&{}&{}   &{}&{}&{}     &2&-1&{}   &{}&{}&{}\\
{}  &{}&{}&{}  &{}&{}&{}   &{}&{}&-1     &-1&2&-1   &{}&{}&{}\\
{}  &{}&{}&{}  &{}&{}&{}   &{}&{}&{}     &{}&-1&2   &-1&{}&{}\\
{}  &{}&{}&{}  &{}&{}&{}   &{}&{}&{}     &{}&{}&-1   &2&-1&{}\\
{}  &{}&{}&{}  &{}&{}&{}   &{}&{}&{}     &{}&{}&{}   &-1&2&-1\\
{}  &{}&{}&{}  &{}&{}&{}   &{}&{}&{}     &{}&{}&{}   &{}&-1&2\\

\end{pmatrix}
\]

\subsection{Wall algorithm approach}

Wall algorithm also provides explicit small central charge realizations of abelian anyon models via $\chi$CFTs.  Inspecting the cases in Thm.  \ref{classification_abelian} and Thm. \ref{K Wall}, we obtain the following explicit Gram matrices $K_A$ for some cases of $A$.
 \begin{itemize}
\item  For $A_{p^r}$, 

\begin{enumerate}
\item when $r$ is odd and $p\equiv-3(mod\,8)$ ,
\[K_{A_{p^r}}=\begin{pmatrix}\frac{4}{p^r}&1&0&0\\
1&\frac{p^r+3}{4}&1&0\\
0&1&2&1\\
0&0&1&2 \\
\end{pmatrix}^{-1}\]

$K_{A_{p^r}}$ has signature 4 and it is positive definite.\\
\item When $r$ is odd and $p\equiv-1(mod\,8)$ ,
\[K_{A_{p^r}}=\begin{pmatrix}\frac{4}{p^r}&1\\
1&\frac{p^r+1}{4}\\
\end{pmatrix}^{-1}\]

$K_{A_{p^r}}$ has signature 2  and it is positive definite.\\ 
\end{enumerate}

\item  For $B_{p^r}$, \begin{enumerate}
\item $r$ odd and  $p\equiv-1(mod\,8)$ we can  choose $m=\frac{p-1}{2}$, then $c=2 m=p-1.$ 

\[K_{B_{p^r}}={\begin{pmatrix} \frac{p-1}{p^r} & 1 &0&\cdots&0&0\\ 1 &  \frac{p^r-1}{p-1}+1&1&\cdots&0&0\\
 0&1&2&\cdots&1&0\\\vdots&\vdots&\vdots&\ddots&\vdots&\vdots\\
 0&0&0&\cdots&2&1\\0&0&0&\cdots&1&2 \end{pmatrix}}^{-1}\]

$K_{B_{p^r}}$ is $(p-1)\times (p-1)$ has signature $p-1$  and it is positive definite.
\item  $r$ odd and  $p\equiv3(mod\,8)$ we can  choose $m=1$, then $c=2 m=2.$ 

\[K_{B_{p^r}}=\begin{pmatrix}\frac{2}{p^r}&1\\
1&\frac{p^r+1}{2}\\
\end{pmatrix}^{-1}\]

$K_{B_{p^r}}$ has signature 2  and it is positive definite.\\

\end{enumerate}
\item For $D_{2^r}, r$ even, $c=3,$ 

\[K_{D_{2^r}}=\begin{pmatrix}\frac{3}{2^r}&1&0\\
1&\frac{2^r+2}{3}&1\\
0&1&2\\
\end{pmatrix}^{-1}\]

$K_{D_{2^r}}$ has signature 3  and it is positive definite.\\

\item For $C_{2^r}, r\equiv0(mod\,4)$ , $c=5,$ 

\[K_{C_{2^r}}=\begin{pmatrix}\frac{5}{2^r}&1&0&0&0\\
1&\frac{2^r+4}{3}&1&0&0\\
0&1&2&1&0\\
0&0&1&2&1\\
0&0&0&1&2
\end{pmatrix}^{-1}\]

$K_{C_{2^r}}$ has signature 5 and it is positive definite.\\
\item 
For $B_{2^r}$,
\begin{enumerate}
\item $r\equiv 1(mod\,3)$ , $c=7,$ 

\[K_{B_{2^r}}=\begin{pmatrix}\frac{7}{2^r}&1&0&0&0&0&0\\
1&\frac{2^r+6}{3}&1&0&0&0&0\\
0&1&2&1&0&0&0\\
0&0&1&2&1&0&0\\
0&0&0&1&2&1&0\\
0&0&0&0&1&2&1\\
0&0&0&0&0&1&2\\

\end{pmatrix}^{-1}\]

$K_{B_{2^r}}$ has signature 7 and it is positive definite.\\
\item $r\equiv0(mod\,3)$ , $c=7,$ 

\[K_{B_{2^r}}=\begin{pmatrix}\frac{7}{2^r}&1&0&0&0&0&0\\
1&\frac{2^r+6}{3}&1&0&0&0&0\\
0&1&2&1&0&0&0\\
0&0&1&2&1&0&0\\
0&0&0&1&2&1&0\\
0&0&0&0&1&2&1\\
0&0&0&0&0&1&2\\

\end{pmatrix}^{-1}\]

$K_{B_{2^r}}$ has signature 7 and it is positive definite.\\

\end{enumerate}

\end{itemize}

\section{Bulk-edge correspondence revisited}

Our reconstruction of $\chi$CFTs for anyon models is motivated by the bulk-edge correspondence in topological phases of matter exemplified by the fractional quantum Hall states.  Our interests lie in broadening the correspondence by taking into account of stability, symmetry, and non-lattice $\chi$CFTs.  Abelian anyon models serve as the simplest examples to understand the general reconstruction.

\subsection{Stability of edge theory}
When abelian anyon models are realized by fractional quantum Hall states, the corresponding boundary $\chi$CFTs provide predictions for experiments in quantum point-contact experiments \cite{CFW96}.   Therefore, it is an interesting problem to understand the stable edge theories \cite{Haldane95}.  Without any extra symmetry protection, we hypothesize that the stability prefers at least small central charges and conformal weights.  Mathematically, we will focus on  
extremal $\chi$CFTs in the minimal genus \cite{TW17}.

\subsubsection{Non-uniqueness in the minimal genus}

As discussed Sec. \ref{conlattice}, the minimal genus can always be realized.  But while  
the trivial abelian anyon model $\mcV$ has a unique realization of the minimal genus $(\mcV ec, 8)$ by the $E_8$ lattice, this uniqueness already fails for the anyon model $\mcA_{23}$ with two interesting $K$-matrices 
$\begin{pmatrix}
2&1\\
1&12
\end{pmatrix}
$
and 
$\begin{pmatrix}
4&1\\
1&6
\end{pmatrix},
$
with the smallest conformal weight of $\frac{1}{23}$ and $\frac{2}{23}$, respectively.  Therefore, $B_{23}$ has two different $\chi$CFTS realizations in the minimal genus $(\mcA_{23}, 2)$.  Of course the full set of $23$ conformal weights are the same as set for both theories mod $\mbbZ$ because they realize the same anyon model.

A special case of the finiteness conjecture \cite{TW17} is:

\begin{cnj}

Given any abelian anyon model $\mcA$, there are only finitely many realizations in each realizable genus $(\mcA,c)$.

\end{cnj}

It would be interesting to understand the classification of extremal VOAs in each realizable genus and how many realizable genera contain extremal VOAs.

\subsection{Equivariant correspondence}

Let $\mcV_B$ be the category of VOAs whose representation category is a unitary modular tensor category or anyon model, and $\mcM_B$ the category of unitary modular tensor categories.  There is a monoidal functor $\mcR: \mcV_B\rightarrow \mcM_B$ that maps a VOA $\mcV$ to its representation category $\textrm{Rep}(\mcV)$.  If $\mcV$ has a symmetry group $G$, then a natural question is if and how we can derive the symmetry group of the modular category of $\textrm{Rep}(\mcV)$ from $G$.

The reconstruction conjecture would imply that $\mcR$ is onto, which is obviously not surjective on symmetries.  For example, for the moonshine VOA $\mcV^{\sharp}$, its representation category is the trivial modular tensor category $\mcV ec$, which can have any group as symmetry. The map $\mcR$ is not injective either as the WZW VOA $SU(2)_1$ has the $SO(3)$ as symmetry \cite{Reh94}, but its representation category the Semion model has trivial topological symmetry group.  

\begin{cnj}

The map $\mcR$ would send the symmetry group $G$ of a VOA $\mcV$ to $G/G_0$ for its representation category $\textrm{Rep}(\mcV)$, where $G_0$ is the identity component of $G$.

\end{cnj}

The symmetry of lattices always gives rise to symmetries of the VOAs involving extensions due to twisting.  Since there are non-lattice realizations, it is not clear how big is the lattice symmetry subgroup.

\subsection{Non-lattice realization}

There are at least $71$ VOAs of central charge=$24$ whose representation category is trivial, while only $24$ of the $71$ are lattice VOAs.  We conjecture that this is not an isolated fact that for each abelian anyon models, there are non-lattice realizations which are not simply product of one with the non-lattice holomorphic VOAs.  A more precise formulation would be the extremal VOAs \cite{TW17}.

\begin{cnj}
For each abelian anyon model $\mcB$, there exists a non-lattice realization by an extremal VOA $\mcV$ such that $\textrm{Rep}(\mcV)\cong \mcB$.
\end{cnj}

Besides the trivial $\mcV ec$, the genus $(\textrm{Semion}, 33)$ has a non-lattice extremal realization \cite{TW17,GT18}.  Moreover, there are only finitely many genera with extremal VOAs for the Semion model \cite{GT18}. 

It would be very interesting to classify non-lattice extemal realizations of abelian anyon models.

\bibliographystyle{plain}

\end{document}